\documentclass[11pt]{amsart}

\usepackage{amsmath,amsthm,indentfirst}
\usepackage{amssymb}
\usepackage{amsfonts}
\usepackage{amscd}
\usepackage[latin1]{inputenc}
\usepackage{hyperref}
\usepackage{ifthen, amsfonts, amssymb, graphicx, srcltx, mathrsfs, xfrac,
amsmath}
\usepackage{enumerate}
\usepackage{pinlabel}
\usepackage{xcolor}

\theoremstyle{plain}
\newtheorem*{maintheorem*}{Main Theorem}

\newtheorem*{thmd*}{Theorem 1.5}
\newtheorem*{thme*}{Theorem 1.6}
\newtheorem*{remark*}{Remark}
\newtheorem*{conjecture*}{Conjecture}
\newtheorem*{prop*}{Proposition}
\newtheorem{thm}{Theorem}[section]
\newtheorem{cor}[thm]{Corollary}
\newtheorem{lem}[thm]{Lemma}
\newtheorem{prop}[thm]{Proposition}

\theoremstyle{definition}

\newtheorem*{proofc*}{Proof of Theorem C}

\newtheorem{definition}[thm]{Definition}

\newtheorem{remark}[thm]{Remark}
\newtheorem{notation}[thm]{Notation}
\newtheorem{claim}[thm]{Claim}

\newcommand{\cl}[1]{\marginpar{\color{purple}\tiny #1 --cl}}
\newcommand{\bm}[1]{\marginpar{\color{red}\tiny #1 --bm}}

\DeclareMathOperator{\CAT}{CAT}
\DeclareMathOperator{\Teich}{Teich}
\DeclareMathOperator{\Mod}{Mod}

\DeclareMathOperator{\WP}{WP}

\DeclareMathOperator{\tw}{tw}

\DeclareMathOperator{\diam}{diam}

\DeclareMathOperator{\nonannular}{non-annular}

\DeclareMathOperator{\pr}{pr}

\DeclareMathOperator{\asya}{\stackrel{+}{\asymp}}
\DeclareMathOperator{\asym}{\stackrel{*}{\asymp}}
\DeclareMathOperator{\la}{\stackrel{+}{\prec}}
\DeclareMathOperator{\lm}{\stackrel{*}{\prec}}

\newcommand{\ML}{\mathcal{ML}}
\newcommand{\PML}{\mathcal{PML}}
\newcommand{\EL}{\mathcal{EL}}

\newcommand{\cC}{\mathcal C}

\newcommand{\cS}{\mathcal S}

\newcommand{\I}{\text{i}}
 
\newcommand{\ep}{\epsilon} 
\newcommand{\OT}{Teichm\"{u}ller}

\numberwithin{equation}{section}

\begin{document}

% \title[short text for running head]{full title}

\title[Limit sets of WP geodesics]{Limit sets of Weil-Petersson geodesics}
%    Only \author and \address are required; other information is
%    optional.  Remove any unused author tags.

\date{\today}

%    author one information
%\author[short name for running head]{full name for first page}

\author[Jeff Brock]{Jeffrey Brock}
\address{Department of Mathematics, Brown University, Providence, RI, }
\email{brock@math.brown.edu}

\author[Chris Leininger]{Christopher Leininger}
\address{ Department of Mathematics, University of Illinois, 1409 W Green ST, Urbana, IL}
\email{clein@math.uiuc.edu}

\author[Babak Modami]{Babak Modami}
\address{ Department of Mathematics, Yale University, 10 Hillhouse Ave, New Haven, CT}
\email{babak.modami@yale.edu}

\author[Kasra Rafi]{Kasra Rafi}
\address{Department of Mathematics, University of Toronto, Toronto, ON }
\email{rafi@math.toronto.edu}

\thanks{The first author was partially supported by NSF grant DMS-1207572, the second author by NSF grant DMS-1510034, the third by NSF grant DMS-1065872, and the fourth
author by NSERC grant \# 435885.}

%    The 2010 edition of the Mathematics Subject Classification is
%    now available.  If you are citing a classification from the
%    new scheme, use the following input coding instead.
%\subjclass[2010]{Primary }

%\keywords{ }

%%%%%%%%%%%%%%%%%%%%%%%%%%%%%% Abstract %%%%%%%%%%%%%%%%%%%%%%%%%%%%%%%%%

\begin{abstract}
In this paper we prove that the limit set of any Weil-Petersson geodesic ray with uniquely ergodic ending lamination is a single point in the Thurston compactification of Teichm\"{u}ller space.  On the other hand, we construct examples of Weil-Petersson geodesics with minimal non-uniquely ergodic ending laminations and limit set a circle in the Thurston compactification. 
% first result is the analogue of a result of Masur for Teichm\"{u}ller geodesics, and the second result is an analogue of various constructions of Teichm\"{u}ller geodesics with limit set a circle by the authors.
\end{abstract}

%%%%%%%%%%%%%%%%%%%%%%%%%%%%%%%%table of contents%%%%%%%%%%%%%%%%%%%%%%%%%%%
\maketitle

\setcounter{tocdepth}{1}
\tableofcontents

%%%%%%%%%%%%%%%%%%%%%%%%%
%%%%%%%%%%%%%%%%%%%%%%%%%%%%%% Introduction %%%%%%%%%%%%%%%%%%%%%%%%%%%%%
\section{Introduction}
%%%%%%%%%%%%%%%%%%%%%%%%%
%%%%%%%%%%%%%%%%%%%%%%%%%

Given a surface $S$, let $\Teich(S)$ denote the Teichm\"uller space of hyperbolic metrics on $S$, and $\Mod(S)$ the mapping class group of $S$.   Thurston compactified $\Teich(S)$ by adjoining the space of projective measured laminations $\PML(S)$, and used this in his classification of elements of $\Mod(S)$; \cite{ThurstonDiff,FLP}.  On the other hand, $\Teich(S)$ has two important, $\Mod(S)$--invariant, unique-geodesic metrics, and hence has natural visual compactifications. 
These metrics have their own drawbacks---the Teichm\"{u}ller metric is not negatively curved \cite{masur:CG} and Weil-Petersson metric is incomplete \cite{wolpert-noncomplete}---and hence the standard results about visual compactification do not readily apply to any of these metrics. For example, the action of $\Mod(S)$ extends continuously to neither the Teichm\"{u}ller visual boundary \cite{Kerck-asymp-teich} nor the Weil-Petersson visual boundary \cite{WPvisualbdry}.

In \cite{2bdriesteich}, Masur showed that the Thurston boundary and the Teichm\"uller visual boundary are not so different, proving that almost every Teichm\"uller ray converges to a point on the Thurston boundary (though positive dimensional families of rays based at a single point can converge to the same point).  Lenzhen \cite{lenzhen} constructed the first examples of Teichm\"uller geodesic rays that do not converge to a unique point in the Thurston boundary, and recent constructions have illustrated increasingly exotic behavior \cite{nonuniqueerg,nue2,nuechaikaetl,2dlimit}.

%A natural question rising then is to determine the limit set of geodesics of a given metric on the Teichm\"{u}ller space in the Thurston compactification. The question for Teichm\"{u}ller geodesics has been already addressed where the answer depends on dynamical properties of the lamination determining the ray. Namely, for uniquely ergodic laminations the limit set is a single point \cite{2bdriesteich}. However, for nonuniquely ergodic laminations one can have different limit sets \cite{nonuniqueerg}\cite{nue2}. 

In this paper we begin an investigation into the behavior of how the Weil-Petersson visual compactification relates to Thurston's compactification.  Specifically, we study the behavior of Weil-Petersson geodesic rays in the Thurston compactification.  Our results are stated in terms of the {\em ending laminations} of Weil-Petersson geodesic rays introduced by Brock, Masur and Minsky in \cite{bmm1}; see \S\ref{subsec : endlam}.  %\bm{I deleted the minimal filling condition from the definition of NUE lamination. I added the condition to the properties of the lamination in the statement of the theorem below}
 The first theorem is a version of Masur's convergence for Teichm\"uller geodesics. We say that a lamination is {\em uniquely ergodic} if it admits a unique transverse measure, up to scaling. Moreover, we say that the lamination is {\em minimal} if every leaf of the lamination is dense in the lamination, and {\em filling} if the lamination intersects every simple closed geodesic on the surface nontrivially. 
%, and is filling (i.e.~transversely intersects every essential simple closed)
%\bm{I deleted the adjective forward for ending lamination of a WP geodesic ray here and other places in the paper}
\begin{thm} \label{thm : UE main vague}
Suppose that the ending lamination of a Weil-Petersson geodesic ray is minimal, filling, and uniquely ergodic.  Then the ray converges in the Thurston compactification to the unique projective class of transverse measures on the ending lamination.
\end{thm}
On the other hand, we prove that there are geodesic rays for which the ending lamination is minimal but non-uniquely ergodic, and whose limit sets are positive dimensional and in fact are non-simply connected.
\begin{thm} \label{thm : NUE main vague}
There exist Weil-Petersson geodesic rays with minimal, filling, non-uniquely ergodic ending laminations whose limit sets in the Thurston compactification are topological circles.
\end{thm}
See Theorem~\ref{thm : limgE 1-sk} for a more precise statement.
Without the minimality assumption, the construction of Weil-Petersson geodesic rays that do not limit to a single point requires some different ideas. This construction is given in \cite{wplimit-nonminimal}.
\medskip

%\bm{here I added a map for the paper}
\subsection{Outline of the paper} Section \ref{sec : background} is devoted to background about \OT\ theory, curve complexes and laminations.
In Section \ref{subsec : seq of curves} we state our technical results about sequences of curves on surfaces that limit to non-uniquely ergodic laminations. These results are minor variations of those in \cite{nue2}, and their proofs are sketched in the appendix of the paper. In Section \ref{sec : sequence of curves on S0p} we construct explicit examples of non-uniquely ergodic laminations on punctured spheres, appealing to the results from Section~\ref{subsec : seq of curves}. In Section \ref{sec : closed geodesics} we study the limiting picture of axes for pseudo-Anosov mapping classes of a punctured sphere arising in the construction of non-uniquely ergodic laminations in Section \ref{sec : sequence of curves on S0p}. In Section \ref{sec : NUE case}, we use this analysis to determine limit sets of our WP geodesic rays 
with non-uniquely ergodic ending laminations, and thus prove Theorem \ref{thm : NUE main vague}. In Section \ref{sec : UE case} we prove Theorem \ref{thm : UE main vague} about limit sets of geodesic rays with uniquely ergodic ending laminations. 
\medskip

\noindent{\bf Acknowledgment:} The authors would like to thank Yair Minsky for useful conversations related to this work.  They would also like to thank the anonymous referee for many helpful suggestions.

%%%%%%%%%%%%%%%%%%
%%%%%%%%%%%%%%%%%%
\section{Background}\label{sec : background}
%%%%%%%%%%%%%%%%%%
%%%%%%%%%%%%%%%%%%
%\bm{I expanded the notation}
\begin{notation}
Our notation for comparing quantities in this paper is as follows: Let $K\geq 1$ and $C\geq 0$. Given two functions $f,g:X\to\mathbb{R}^{\geq 0}$, we write $f\asymp_{K,C}g$ if 
$\frac{1}{K}g(x)-C\leq f(x)\leq Kg(x)+C$ for all $x\in X$, $f\stackrel{+}{\asymp}_{C}g$ if $g(x)-C\leq f(x)\leq g(x)+C$ and $f\asym_{K}g$ if $\frac{1}{K}g(x)\leq f(x)\leq Kg(x)$.
The notation $f\la_C g$ means that $f(x)\leq g(x)+C$ for all $x\in X$ and $f\lm_K g$ means that $f(x)\leq Kg(x)$ for all $x\in X$.

When the numbers $K,C$ are understood from the context we drop them from the notation.  
\end{notation}

\subsection{Surfaces and subsurfaces:} In this paper surfaces are connected, orientable and of finite type with boundaries or punctures. We denote a surface with genus $g$ and $b$ boundary curves or punctures by $S_{g,b}$ and define the complexity of the surface by $\xi(S_{g,b})=3g+b-3$. The main surfaces we consider always have only punctures, however we consider subsurfaces of the main surfaces with both punctures and boundary curves.

%%%%%%%%%%%%%%%%%%%%%%%
\subsection{Curves and laminations}\label{subsec : combinatorics}
%%%%%%%%%%%%%%%%%%%%%%%
\begin{notation}
  Throughout this paper, by a {\em curve} we mean the isotopy class
  of an essential, simple, closed curve.  When convenient, we do not distinguish between a curve and a representative of the isotopy class.
    A {\em multicurve} is a collection of pairwise disjoint curves (that is, curves with pairwise disjoint representatives). 
    
    By a
  {\em subsurface} of $S$, we mean the isotopy class of an
  embedded essential subsurface: one whose boundary consists of
  essential curves and whose punctures agree with those of $S$.

  We say that two curves $\alpha$ and $\beta$ overlap and denote it by
  $\alpha\pitchfork \beta$ if the curves $\alpha$ and $\beta$ cannot
  be represented disjointly on the surface $S$. Two multicurves $\sigma$
  and $\tau$ overlap if there are curves $\alpha\in\sigma$ and
  $\beta\in\tau$ which overlap.
   A curve $\alpha$ and a subsurface $Y\subseteq S$ overlap, denoted
   by $\alpha\pitchfork Y$, if $\alpha$ cannot be realized
  disjointly from $Y$ (up to homotopy). A multicurve and a
  subsurface overlap if a component of the multicurve overlaps with the subsurface. We say that two subsurfaces $Y$ and
  $Z$ overlap, and denote it by $Y\pitchfork Z$, if
  $\partial{Y}\pitchfork Z$ and $\partial{Z}\pitchfork Y$.
\end{notation}

We refer the reader to \cite{mm1,mm2} for
 background about the curve complex and subsurface projection maps. Denote the curve complex of a surface $S$  
by $\mathcal{C}(S)$ and the set of vertices of the complex by $\mathcal{C}_0(S)$. 
The set $\mathcal{C}_0(S)$ is in fact the set of essential simple closed curves on $S$.

% \bm{added the definition of pants decomposition and marking}
 A {\em pants decomposition} on the surface $S$ is a multicurve with a maximal number of components.
 A {\em (partial) marking} $\mu$ on the surface consists of pants decomposition, called the base of $\mu$, and a choice transversal curves for 
 (some) all curves in the base.
   For background about pants and marking graphs and hierarchical structures and (hierarchy)
resolution paths in pants and marking graphs we refer the reader to
\cite{mm2}. We denote the pants graph of the surface $S$ by $P(S)$. Here we only recall that hierarchy paths are certain quasi-geodesics
in $P(S)$ with quantifiers that only depend on the topological type of $S$. 

Let $Y\subseteq S$ be an essential subsurface. The {\em $Y$--subsurface projection coefficient} of two multicurves,
markings or laminations $\mu,\mu'$ is 
defined by
\begin{equation}d_Y(\mu,\mu'):=\diam_{\cC(Y)}\Big(\pi_Y(\mu)\cup\pi_Y(\mu')\Big).\end{equation}
Here $\pi_Y$ is the subsurface projection (coarse) map and $\diam_{\cC(Y)}(\cdot)$ denotes the diameter of the given subset of $\mathcal{C}(Y)$.
%\bm{the notation $d_\gamma$ is introduced here}
When $Y$ is an annular subsurface with core curve $\gamma$ we also denote $d_{Y}(\mu,\mu')$ by $d_\gamma(\mu,\mu')$.

Our results in this paper are formulated in terms of subsurface coefficients which can be thought of as an analogue of
continued fraction expansions which provide a kind of symbolic
coding for us. 

We assume that the reader is familiar with basic facts
about laminations and measured laminations on hyperbolic surfaces (see
e.g. \cite{phtraintr} for an introduction). We denote the space of measured laminations equipped with the weak$^{*}$ topology by $\mathcal{ML}(S)$ 
and the space of projective classes of measured laminations equipped with the quotient topology by $\mathcal{PML}(S)$.

Recall that a measurable lamination is
{\em uniquely ergodic} if it supports exactly one transverse measure up to scale. Otherwise, the lamination
is {\em non-uniquely ergodic}. 

An important property of
curve complexes is that they are Gromov hyperbolic \cite{mm1}. By the
result of Klarreich \cite{bdrycc} the Gromov boundary of the curve
complex is homeomorphic to the quotient space of the space of (projective) measured
laminations with minimal, filling supports by the measure forgetting map
equipped with the quotient topology, denoted by $\EL(S)$.
\medskip
 
%In the sequel we recall some of the basic results in the context of
%curve complexes we need in this paper.
 
The Masur-Minsky distance formula \cite{mm2} provides a coarse
estimate for the distance between two pants decompositions in the pants 
graph $P(S)$. More precisely, there exists a constant $M>0$ depending on the
topological type of $S$ with the property that for any threshold
$A\geq M$ there are constants $K\geq 1$ and $C\geq 0$ so that for any
$P,Q\in P(S)$ we have
 \begin{equation}\label{eq : MM dist}
 d(P,Q) \asymp_{K,C} \sum_{{\substack{ Y\subseteq S \\ \nonannular }}} \{d_{Y}(P,Q)\}_A
\end{equation} 
where the cut-off function is defined by $\{x\}_A=\begin{cases}x &\mbox{if}\; x\geq A \\ 0 &\mbox{if}\; x\leq A\end{cases}$.

 The following theorem is a straightforward consequence of \cite[Theorem 3.1]{mm2}.
 
 \begin{thm}\textnormal{(Bounded geodesic image)}\label{thm : bddgeod}
Given $K\geq 1$ and $C\geq 0$. Suppose that $\{\gamma_i\}_i$ is a sequence of curves which forms a $1-$Lipschitz $(K,C)$--quasi-geodesic in $\mathcal{C}(S)$. 
Further, suppose that for a subsurface $Y\subsetneq S$, $\gamma_i\pitchfork Y$ for all $i$. Then there is a constant $G>0$ depending on $K,C$ so that
 \[\diam_{\cC(Y)}\Big(\{\pi_{Y}(\gamma_i)\}_i\Big)\leq G.\]
 \end{thm}
 
 For the following inequality, see \cite{beh}\cite{Mangahasbehineq}.

\begin{thm}\textnormal{(Behrstock inequality)}\label{thm : beh ineq}
There exists a constant $B_0>0$ such that for any two subsurfaces
$Y,Z\subsetneq S$ with $Y\pitchfork Z$ and a fixed marking or lamination $\mu$ we have that
\[\min\Big\{d_{Y}(\partial{Z},\mu),d_{Z}(\partial{Y},\mu)\Big\}\leq B_0.\]
\end{thm}

We also need the following {\em no back-tracking} property of hierarchy paths, which follows from inequality (6.3) in \cite[\S 6.3]{mm2}.
\begin{thm}\label{thm : no back tracking}
There is a constant $C>0$ depending only on the topological type of $S$ so that given a hierarchy path $\varrho:[m,n]\to P(S)$ ($[m,n]\subset \mathbb Z\cup\{\pm\infty\}$), for parameters $i_1\leq i_2\leq i_3\leq i_4$ in $[m,n]$, and a non-annular subsurface $Y\subseteq S$ we have
\[d_Y(\varrho(i_1),\varrho(i_4))\geq d_Y(\varrho(i_2),\varrho(i_3))-C.\]
\end{thm}

%\cl{move this to later?  We don't use it until section 5.}

\subsection{Twist parameter}
We define the twist parameter of a curve $\delta$ about $\gamma$ at a point $X$ in Teichm\"{u}ller space by
\begin{equation}\label{eq : twist}\tw_{\gamma}(\delta,X):=d_{\gamma}(\mu,\delta)\end{equation}
where $\mu$ is a Bers marking at $X$ (for definition of Bers marking see $\S$\ref{subsec : WP}). 

Note that for a filling set of bounded length curves $\Gamma$ at $X$ we have $\tw_{\gamma}(\delta,X)\asya \diam_{\cC(\gamma)}(\Gamma \cup \delta)$.

\subsection{The Thurston compactification}\label{subsec : Thcpct}
Recall that a point in the Teichm\"{u}ller space $\Teich(S)$ is a {\em marked complete hyperbolic surface} $[f \colon S \to X]$. The mapping class group of $S$, denoted by $\Mod(S)$, acts by remarking on $\Teich(S)$ and the quotient is the moduli space of hyperbolic surfaces $\mathcal M(S)$. 

Given a curve $\alpha\in\cC_0(S)$, the hyperbolic length of $\alpha$ at $[f \colon S \to X]$ is defined to be the hyperbolic length of the geodesic homotopic to $f(\alpha)$ in $X$. Abusing notation and denoting the point in $\Teich(S)$ by $X$, we write the hyperbolic length simply as $\ell_{\alpha}(X)$. For an $\ep>0$, the {\em $\ep$--thick part} of \OT\ space consists of points $X\in \Teich(S)$ with $\ell_\alpha(X)\geq 2\ep$ for all curves $\alpha$. The projection of this set to the moduli space is the $\ep$--thick part of moduli space.

 The hyperbolic length function extends to a continuous function
\[ \ell_{\cdot}(\cdot) \colon \Teich(S) \times \mathcal{ML}(S) \to \mathbb R.\]

Let $\nu$ be a measurable lamination and $\bar\nu$ a measured lamination with support $\nu$. Moreover, denote the projective class of $\bar\nu$ by $[\bar\nu]$. The Thurston compactification, $\widehat{\Teich(S)} = \Teich(S) \cup \mathcal{PML}(S)$ is constructed so that a sequence $\{X_n\}_n \subseteq \Teich(S)$ converges to $[\bar \nu] \in \mathcal{PML}(S)$ if and only if
\[ \lim_{n \to \infty} \frac{\ell_\alpha(X_n)}{\ell_\beta(X_n)} = \frac{\I(\alpha,\bar \nu)}{\I(\beta,\bar \nu)}\]
for all simple closed curves $\alpha,\beta$ with $\I(\bar \nu,\beta) \neq 0$.
Here and throughout this paper the bi-homogenous function $\I(\cdot,\cdot)$ denotes the geometric intersection number of 
two curves and its extension to the space of measured laminations $\mathcal{ML}(S)$. See \cite{bonlam,FLP} for more details on the intersection function and Thurston compactification.

%%%%%%%%%%%%%%%%%%%%%%%%%
\subsection{Sequences of curves} \label{subsec : seq of curves}
%%%%%%%%%%%%%%%%%%%%%%%%%

In \cite{nonuniqueerg} and \cite{nue2} the authors studied infinite sequences of curves on a surface that limit to non-uniquely ergodic laminations.  The novelty in this work is that local estimates on subsurface projections and intersection numbers are promoted to global estimates on these quantities.
We require minor modifications of some of the key results from \cite{nue2} so as to be applicable to the sequences of curves on punctured spheres described in \S\ref{sec : sequence of curves on S0p}.  We state the results here and sketch their proofs in the appendix for completeness. 

Given a curve $\gamma$ let $D_\gamma$ be the positive (left) Dehn twist about $\gamma$.

\begin{definition}\label{def : sequence of curves} 
Fix positive integers $m\leq \xi(S)$ and $b' \geq b >0$, and a sequence $\mathcal E = \{e_k\}_{k=0}^\infty\subseteq \mathbb{N}$.
%, also a real number $a>1$ and a sequence $\mathcal E$ satisfying (\ref{eq : exp}) with factor $a$. 
We say that a sequence of curves $\{\gamma_k\}_{k=0}^\infty$ satisfies $\mathcal P = \mathcal P(\mathcal E)$ if the following hold:
\begin{enumerate}[(i)] 
\item any $m$ consecutive curves are pairwise disjoint,
\item any consecutive $2m$ curves fill $S$, and
\item for all $k \geq m$, $\gamma_{k+m} = D^{e_k}_{\gamma_k}(\gamma_{k+m}')$, where $\gamma'_{k+m}$ is a curve such that
\[ \I(\gamma'_{k+m},\gamma_j) \left\{ \begin{array}{ll} \leq b' & \mbox{ for $j=k-m,\ldots,k+m-1$}\\
= b & \mbox{for } j = k,k-1\\
= 0 & \mbox{for } j=k+1,\ldots,k+m-1 \end{array}\right. \]
%\[ \I(\gamma'_{k+m},\gamma_j) \left\{ \begin{array}{ll} \in [1, b_2] & \mbox{ for $j=k+m-\ell-1,\ldots,k-1$}\\ 
%\[ \I(\gamma'_{k+m},\gamma_j) \left\{ \begin{array}{ll} \in [b_1,b_2] & \mbox{for some } j\in\{k - m,\ldots,k-1\} \;\mbox{and $0$ for the rest} \\
%= b & \mbox{for } j = k\\
%= 0 & \mbox{for } j=k+1,\ldots,k+m-1 \;\mbox{and}\\
%& j=k-m,\ldots,k-m+\ell  \end{array}\right. \]
\end{enumerate}
\end{definition}

\begin{remark}
The only real difference between this definition and the one given in our previous paper \cite[\S 3]{nue2} is that this one requires fewer of the intersection numbers to be nonzero.
%This results in some weaker estimates, but fortunately, these estimates still suffice for our purposes.
\end{remark}

For the remainder of this subsection, we will assume that $\Gamma = \{\gamma_k\}_k$ satisfies $\mathcal P$ for some $m,b,b'$ and $\mathcal E = \{e_k\}_k$  (see \S\ref{sec : sequence of curves on S0p} for explicit examples).  Furthermore, we assume there exists an $a \geq 1$ such that $e_{k+1} \geq a e_k$ for all $k\geq 0$.

The first result describes the behavior of $\{\gamma_k\}_k$ in the curve complex of $S$ and its subsurfaces.  Let $\mathbb M$ be the monoid generated by $m$ and $m+1$, that is 
$$\mathbb M=\{im+j(m+1)\mid i,j\in\mathbb Z^{\geq 0}\}.$$

\begin{thm}\label{thm : subsurface coeff estimate} There exist constants $E,K,C >0$ such that if $e_0 \geq E$, then $\{\gamma_k\}_k$ is a $1$--Lipschitz $(K,C)$--quasi-geodesic.  In particular, there exists a $\nu\in\EL(S)$ such that any accumulation point of $\{\gamma_k\}_k$ in $\PML(S)$ is supported on $\nu$.

Furthermore, there exists a constant $R> 0$, and for a marking $\mu$ another constant $R(\mu)> 0$ depending on $\mu$, so that for $i < k < j$, with $k-i,j-k \in \mathbb M$, we have $\gamma_i \pitchfork \gamma_k$, $\gamma_j \pitchfork \gamma_k$, and
\begin{equation}\label{eqn : gamma_i, big projection}
d_{\gamma_k}(\gamma_i,\gamma_j),d_{\gamma_k}(\gamma_i,\nu) {\asya}_R e_i \quad \mbox{ and } \quad d_{\gamma_k}(\mu,\gamma_j),d_{\gamma_k}(\mu,\nu){\asya}_{R(\mu)} e_i.
\end{equation}
Also, for any $i < j$ and a subsurface $W$ which is not an annulus with core curve $\gamma_k$ for some $k$ we have
\begin{equation}\label{eqn : not gamma_i, small projection}
d_W(\gamma_i,\gamma_j),d_W(\gamma_i,\nu) \leq R \quad \mbox{ and } \quad d_W(\mu,\gamma_j),d_W(\mu,\nu) \leq R(\mu).
\end{equation}
\end{thm}

The next result provides estimates on intersection numbers for curves in our sequence.
To describe the estimates, for all $i < k$, define the integers
 \begin{equation}\label{eq : A(i,k)} A(i,k):= \prod_{\substack{i + m\leq j< k, \\j\equiv k\mod m}} be_j  \end{equation}
where the product is taken to be $1$ whenever the index set is empty. 

\begin{thm} \label{thm : intersection}
If $a > 1$ is sufficiently large and $e_{k+1} \geq a e_k$, then there exists $\kappa_0 \geq 1$ such that $ \I(\gamma_i,\gamma_k) \leq \kappa_0 A(i,k)$
for all $i < k$, and
\begin{equation}\label{eq : gigk}\I(\gamma_i,\gamma_k)\asym_{\kappa_0} A(i,k).\end{equation}
if $k - i  \geq 2m$ and $i \equiv k$ mod $m$, or if $i \leq 2m-1$ and  $k - i \geq m^2 + m -1$.  

For any curve $\delta$, there exists $\kappa(\delta) \geq 1$ such that for all $k$ sufficiently large
\begin{equation}\label{eq : dgk}\I(\delta,\gamma_k)\asym_{\kappa(\delta)} A(0,k).\end{equation}
\end{thm}

For reference, we also record the following simple fact (see \cite[Lemma~5.6]{nue2}).
\begin{lem} \label{lem : intersection estimate ratios} Suppose that $\mathcal E = \{e_k\}_k$ satisfies $e_k \geq ae_{k-1}$ for all $k$ and some $a > 1$.  Then whenever $k < l$, we have
\[ \frac{A(i,k)}{A(i,l)} \leq a^{1-\lfloor \frac{l-i}{m} \rfloor}.\]
\end{lem}
%Basically, the point is that the product defining $A(i,l)$ has at least as many terms as that of $A(i,k)$, and each term of $A(i,l)$ is at least $a$ times the corresponding term in $A(i,k)$.

The final result tells us that $\{\gamma_k\}_k$ splits into $m$ subsequences, each projectively converging to a distinct ergodic measure on $\nu$.

\begin{thm} \label{thm : nue}
If $\{\gamma_k\}_k$ is as in Theorems~\ref{thm : subsurface coeff estimate} and \ref{thm : intersection}, then the sequence determines a $\nu\in\EL(S)$
 which is non-uniquely ergodic and supports $m$ ergodic measures, $\bar\nu^0,\ldots,\bar\nu^{m-1}$, given by
\[
\lim_{i\to\infty}\frac{\gamma_{h+im}}{A(0,h+im)}=\bar{\nu}^{h},
\]
for each $h = 0,\ldots,m-1$.
\end{thm}

%%%%%%%%%%%%%%%%%%%%%%%%%
\subsection{Weil-Petersson metric}\label{subsec : WP}
%%%%%%%%%%%%%%%%%%%%%%%%%

The Weil-Petersson (WP) metric is an incomplete, mapping class group invariant, Riemannian metric with negative curvature on the Teichm\"{u}ller space. The WP completion of Teichm\"{u}ller space $\overline{\Teich(S)}$ is a stratified $\CAT(0)$ space. Each stratum is canonically isometric to the product of Teichm\"{u}ller spaces of lower complexity, each equipped with the WP metric. More precisely, for any possibly empty multicurve $\sigma$ on $S$ the stratum $\cS(\sigma)$ consists of finite type Riemann surfaces appropriately marked by $S\backslash\sigma$, and this is isometric to the product of Teichm\"uller spaces of the connected components of $S\backslash\sigma$.  An important property of completion strata is the following {\em non-refraction} property.

\begin{thm}\textnormal{(Non-refraction; \label{thm : non-ref}\cite{wol}\cite{dwwp})}
The interior of the geodesic segment connecting a point $X\in\cS(\sigma)$ to a point $Y\in\cS(\sigma')$ lies in the stratum $\cS(\sigma\cap\sigma')$.
\end{thm}

Let $L_S>0$ be the {\em Bers constant} of $S$ (see \cite[\S 5]{buser}). Then each point $X\in\Teich(S)$ has a pants decomposition $P$ ({\em Bers pants decomposition}) with the property that the length of every curve in $P$ with respect to $X$ is at most $L_S$. Any curve in a Bers pants decomposition is called a {\em Bers curve}. Moreover, a marking whose base is a Bers pants decomposition and has transversal curves with shortest possible length is called a {\em Bers marking}. 

Brock \cite{br} showed that the coarse map 
\begin{equation}\label{eq : Brock qi}Q:\Teich(S)\to P(S)\end{equation}
which assigns to a point in the Teichm\"{u}ller space a Bers pants decomposition at that point is a quasi-isometry.
\medskip

Using the non-refraction property of completion strata Wolpert
\cite{wols} gives a picture for the limits of sequences of bounded
length WP geodesic segments in Teichm\"{u}ller space after
remarkings. The following strengthening of the picture was proved
in \cite[\S 4]{wpbehavior}.  For reference, given a multicurve
$\sigma$ let $\tw(\sigma) < \Mod(S)$ denote the subgroup
generated by positive Dehn twists about the curves in $\sigma$.
%\bm{I added the partition to the statement of theorem}
\begin{thm}\textnormal{(Geodesic Limit)}\label{thm : geodesic limit}
 Given $T>0$, let  $\zeta_{n}:[0,T]\to \Teich(S)$ be a sequence of
 geodesic segments parametrized by arc length. After possibly passing to a subsequence, there exist
 a partition $0=t_0<\ldots<t_{k+1}=T$ of the interval $[0,T]$,
 multicurves $\sigma_l, \; l=0,\ldots,k+1$, a multicurve $\hat{\tau}$
 with $\hat{\tau}=\sigma_l\cap\sigma_{l+1}$ for all $l=0,\ldots,k$ and a
 piecewise geodesic segment 
\[\hat{\zeta}:[0,T]\to\overline{\Teich(S)}\]
such that
\begin{enumerate}
\item $\hat{\zeta}(t_l)\in\cS(\sigma_l) \ \ \text{for each} \ \ l=0,\ldots,k+1,$
\item $\hat{\zeta}((t_l,t_{l+1}))\subset \cS(\hat{\tau}) \ \ \text{for
  each} \ \ l=0,\ldots,k,$ 
  \item there exist elements $\psi_{n} \in \Mod(S)$ which are either trivial or unbounded as $n \to \infty$ and elements
$\mathcal{T}_{l,n}\in\tw(\sigma_l-\hat{\tau})$ 
% where
% $\tw(\sigma_l-\hat{\tau})$ is the subgroup of the mapping class group
% generate by powers of positive Dehn twists about the curves in
% $\sigma_l-\hat{\tau}$ 
such that for any $\gamma\in\sigma_l-\hat{\tau}$
the power of the positive Dehn twist $D_{\gamma}$ about $\gamma$ is unbounded as $n\to\infty$, and we have:
$$\lim_{n\to\infty}\psi_{n}(t)=\hat{\zeta}(t)$$ 
for any $t\in[0,t_1]$. Moreover, for each $l=1,\ldots, k$ let
\begin{equation}\label{eq : varphi}\varphi_{l,n}=\mathcal{T}_{l,n}\circ \ldots\circ \mathcal{T}_{1,n}\circ \psi_{n},\end{equation}
 then
 \[\lim_{n\to\infty}\varphi_{l,n}(\zeta_{n}(t))=\hat{\zeta}(t)\] 
 for any $t\in [t_l,t_{l+1}]$.
 \end{enumerate}
\end{thm}

%\cl{get more precise statements of next few facts (including converse of Theorem \ref{thm : tw-sh}) needed to prove statements in section 5.}

%\marginpar{\tiny I added the converse. The theorem is used in the proof of lemma 4.5. -B}

In \cite{wpbehavior} controls on length-functions along WP geodesics in terms of
subsurface coefficients are developed. The following are corollaries 4.10 and 4.11 in \cite{wpbehavior}. Here we denote a Bers marking at a
point $X\in\Teich(S)$ by $\mu(X)$. 

\begin{thm} \label{thm : tw-sh}
Given $T, \epsilon_{0}$ and $\epsilon<\epsilon_{0}$ positive, there is
an $N\in\mathbb{N}$ with the following property. Let  $\zeta: [0, T']
\to \Teich(S)$ be a WP geodesic segment parametrized by arc length, of length $T'\leq T$, such that 
\[\sup_{t \in [0,T']}\ell_{\gamma}(\zeta(t))\geq \epsilon_{0}.\] 
Then if $d_{\gamma}(\mu(\zeta(0)), \mu(\zeta(T'))) >N$ we have 
  \[\inf_{t \in [0,T']} \ell_{\gamma}(\zeta(t))\leq \epsilon.\]
\end{thm}
 
 \begin{thm} \label{thm : shtw}
Given $T,\epsilon_{0},s$ positive with $T>2s$ and $N\in\mathbb{N}$,
there is an $\epsilon\in(0,\epsilon_{0})$ with the following property. Let
$\zeta: [0, T'] \to\Teich(S)$ be a WP geodesic segment parametrized by
arc length of length $T'\in [2s,T]$. Let $J \subseteq [s,T'-s]$ be a
subinterval. Suppose that for some $\gamma\in\mathcal{C}_{0}(S)$ we
have 
\[\sup_{t \in [0,T']} \ell_{\gamma}(\zeta(t)) \geq \epsilon_{0}.\]
Then, if $\inf_{t \in J} \ell_{\gamma}(\zeta(t)) \leq \epsilon$, we have 
\[d_{\gamma}(\mu(\zeta(0)),\mu(\zeta(T')))>N.\]
\end{thm}

%\bm{I added the corollry}

In this paper we will frequently use the following result of Wolpert for estimating distance of a point and a completion stratum. 
It is part of \cite[Corollary 4.10]{wolb}.

\begin{cor}\label{cor : dist to stratum}
Let $X\in\Teich(S)$ and let $\sigma$ be a multicurve, then
\[ d_{\WP}(X,\cS(\sigma))\leq \sqrt{2\pi\sum_{\alpha\in\sigma}\ell_{\alpha}(X)} \;\;\text{and}\;\;\]
\[d_{\WP}(X,\cS(\sigma))=\sqrt{2\pi\sum_{\alpha\in\sigma}\ell_{\alpha}(X)}+O\Big(\sum_{\alpha\in\sigma}\ell_{\alpha}(X))^{5/2}\Big)\]
where the constant of the $O$ notation depends only on an upper bound for the length of curves in $\sigma$ at $X$.
\end{cor}

\subsection{End invariants}\label{subsec : endlam} 
%\bm{I slightly adjusted this part}
 Brock, Masur, and Minsky \cite{bmm1} introduced the notion
 of {\em ending lamination} for Weil-Petersson geodesic rays as follows:
 Note that here and throughout the paper all geodesics would be parametrized by the arc length.
  Let $r \colon [a,b)\to\Teich(S)$ be a complete WP geodesic ray (a ray whose domain cannot be extended to the left end point $b$). 
 First, the
 weak$^{*}$ limit of an infinite sequence of weighted distinct Bers
 curves at times $t_i\to b$ is an {\em ending measure} of the ray
 $r$, and any curve $\alpha$ with $\lim_{t\to
   b}\ell_{\alpha}(r(t_i))=0$ is a {\em pinching curve} of $r$. Now the
 union of supports of ending measures and pinching curves of $r$ is
 the {\em ending lamination} of $r$ which we denote by $\nu(r)$. 

 % \cl{Do we need all this about forward and backward ending laminations?  Remove if not.}
Let $g \colon I \to \Teich(S)$ be a WP geodesic, where $I\subseteq \mathbb{R}$ is an interval. 
Denote the left and right end points of $I$ by $a,b$, respectively, and let $c$ be a point in the interior of $I$.
 If $g$ is extendable to $b$ in $\overline{\Teich(S)}$, including the situation that $b\in I$, then the forward end invariant of $g$, denoted by $\nu^+$,
 is a (partial) Bers marking at $g(b)$. If not, the forward end invariant of $g$ 
 (also called the forward ending lamination) is the ending lamination of the geodesic ray $g(t)|_{[c,b)}$ defined above.
  The backward end invariant (ending lamination) $\nu^-$ of $g$ is defined similarly considering the ray $g(-t)|_{(a,c]}$. 
  Finally, the pair $(\nu^-,\nu^+)$ is called the end invariant of $g$.  
  
  For example, the end invariant of a geodesic segment $g\colon [a,b]\to \Teich(S)$ is the pair of markings $(\mu(g(a)),\mu(g(b)))$.

 For more detail about end invariants of WP geodesics and their application to study the geometry and dynamics of Weil-Petersson metric see \cite{bmm1,bmm2}\cite{wpbehavior,asympdiv}\cite{wprecurnue}\cite{Hamen-teich-wp}.

\subsection{Bounded combinatorics} Given $R>0$, a pair of (partial)
markings or laminations $(\mu,\nu)$ has {\em $R$--bounded
  combinatorics} if for any proper subsurface $Y\subsetneq S$ the
bound 
\begin{equation}\label{eq : bdd comb}
d_{Y}(\mu,\nu)\leq R
\end{equation}
holds. If the bound holds only for non-annular subsurfaces of $S$ we say that the pair has {\em non-annular $R$--bounded combinatorics}. 
\medskip

The following theorem relates the non-annular bounded combinatorics of end invariants to the behavior of WP geodesics. 

\begin{thm}\label{thm : bd comb recurrence}
For any $R>0$ there is an $\ep>0$ so that any WP geodesic ray $r \colon [0,\infty) \to \Teich(S)$ whose end invariant has non-annular $R$--bounded combinatorics visits the $\ep$--thick part of $\Teich(S)$ infinitely often.
\end{thm}
\begin{proof}
The fact that an individual ray $r$ visits an $\ep$--thick part of $\Teich(S)$ infinitely often is \cite[Theorem 4.1]{wprecurnue}. To show that $\ep$ can be chosen uniformly for all geodesic rays $r$ whose end invariants have non-annular $R$--bounded combinatorics consider a decreasing sequence $\ep_n\to 0$ and a sequence of WP geodesic rays $r_n:[0,\infty)\to\Teich(S)$ with non-annular $R$--bounded combinatorics end invariants $(\mu(r_n(0)),\nu^+_n)$ and assume that $\ep_n$ is the largest number that $r_n$ visits the $\ep_n$--thick part of \OT\ space infinitely often. In particular, for each $n\geq 1$ there is a time $t_n$ so that $r_n([t_n,\infty))$ does not intersect the $2\ep_n$--thick part of $\Teich(S)$.

 Since the end invariant of $r_n$, has non-annular $R$--bounded combinatorics, a hierarchy path $\varrho_n$ between the end invariant is stable in the pants graph $P(S)$ \cite[Theorem 4.3]{bmm2}\cite[Theorem 5.13]{wpbehavior}, in particular $\varrho_n$ and $Q(r_n)$, the image of $r_n$ under Brock's quasi-isometry (\ref{eq : Brock qi}), $D$--fellow travel in $P(S)$ where the constant $D$ depends only on $R$. Theorem \ref{thm : no back tracking} then guarantees that any two points along $\varrho_n$ also satisfy the non-annular bounded combinatorics condition (\ref{eq : bdd comb}) with a larger constant. 
 
 Now for any two times $t_1,t_2\in[0,\infty)$, let $i_1,i_2$ be so that $d(\varrho_n(i_1),Q(r_n(t_1)))$ and $d(\varrho_n(i_2),Q(r_n(t_2)))$ are at most $D$. Then from the distance formula (\ref{eq : MM dist})
 we see that all subsurface coefficients of the pair $(\varrho_n(i_1),Q(r_n(t_1)))$ and the pair $(\varrho_n(i_2),Q(r_n(t_2)))$ are bounded by $\max\{A,KD+KC\}$ for a choice of threshold the $A$ in the formula. This together with the fact that the pair $(\varrho_n(i_1),\varrho_n(i_2))$ has non-annular bounded combinatorics imply that the pair $(Q(r_n(t_1)),Q(r_n(t_2)))$ also satisfies the non-annular bounded combinatorics condition for some $R'>R$ which depends only on $R$. 
 
For $R'>0$ as above let the constants $T_0>0$ and $\ep_0>0$ be as in Lemma 4.2 of \cite{wprecurnue}. Let $I_n\subset [t_n,\infty)$ be an interval of length $T_0$ which contains a time $s_n$ so that $r_n(s_n)$ is in the $\ep_n$--thick part of \OT\ space and $s_n\to\infty$ as $n\to\infty$. By the previous paragarph the end points of $r_n|_{I_n}$ have non-annular $R'$--bounded combinatorics, so by Lemma 4.2 of \cite{wprecurnue} after possibly passing to a subsequence the geodesic segments $r_n|_{I_n}$ intersect the $\epsilon_0$--thick part of \OT\ space. But for $n$ sufficiently large $2\ep_n<\ep_0$ and by the choice of $I_n$ the segment $r_n|_{I_n}$ does not intersect the $2\ep_n$--thick part, which contradicts that $r_n|_{I_n}$ intersects the $\ep_0$--thick part. The existence of a uniform $\ep$ for all $r$ with non-annular $R$--bounded combinatorics follows from this contradiction.
\end{proof}

%%%%%%%%%%%%%%%%%%%%%%%%%
\subsection{Isolated annular subsurfaces}\label{subsec : isolated annular subsurf}
%%%%%%%%%%%%%%%%%%%%%%%%%

In this section we recall the relevant aspects of the notion of an
isolated annular subsurface along a hierarchy path from \cite[\S 6]{wpbehavior}
and its consequences for our purposes in this paper.  

Let $(\mu,\nu)$ be a pair of (partial) markings or laminations with non-annular $R$--bounded combinatorics. A hierarchy path
$\varrho:[m,n]\to P(S)$, $[m,n]\subseteq \mathbb Z\cup\{\pm\infty\}$, with end points $(\mu,\nu)$ is stable in the pants graph of $S$ \cite{bmm2}. In particular, $Q(g)$ the image of a WP
geodesic $g$ with end invariant $(\mu,\nu)$ under Brock's
quasi-isometry (\ref{eq : Brock qi}) $D$--fellow travels $\varrho$ for a $D>0$ depending only on $R$. 
For a parameter $i\in[m,n]$, we say that the time $t$ corresponds to $i$, if $Q(g(t))$ is within distance $D$ of $\varrho(i)$, and vice versa.

Let $i\in [m,n]$, and let $Q$ be a pants decomposition within distance $D$ of the point $\varrho(i)$, moreover let $\gamma$ be a curve in $Q$.
 By \cite[Definition 6.3]{wpbehavior} the annular subsurface with core curve $\gamma$ is
isolated at $i$ along $\varrho$ and hence by \cite[Lemma 6.4]{wpbehavior} we have:
%\cl{The referee is confused about this paragraph preceding the lemma, I think because (a) "isolated" has not be defined, and (b) the reference is to a definition.  Rephrase to clarify.}
\begin{lem} \label{lem : ann coeff comp}\textnormal{(Annular
    coefficient comparison)} 
There are positive constants ${\bar w},b$ and $B$ depending on $R$ and
a constant $L$ depending only on the topological type of $S$, so that
for the curve $\gamma$, a time $t$ corresponding to $i$, any $s\geq {\bar w}$ and $s'\asya_b s$, we have:
\begin{equation}\label{eq : anncoeffcomp}
d_{\gamma}(\mu(g(t-s')),\mu(g(t+s'))) \asya_{B}
  d_{\gamma}(\varrho(i-s),\varrho(i+s)),
  \end{equation} 
where $\mu(\cdot)$ is a choice of Bers marking at the point. Moreover, 
\begin{equation}\label{eq : lng}\min \{ \ell_{\gamma}(g(t-s')),\ell_{\gamma}(g(t+s'))\} \geq L.\end{equation}
\end{lem}
%\cl{Following referee's request, change ${\bf w}$ to a non-bold face constant here and below where it is used.}

%%%%%%%%%%%%%%%%%%%%%%%%
%%%%%%%%%%%%%%%%%%%%%%%%
\section{Sequences of curves on punctured spheres}\label{sec : sequence of curves on S0p}
%%%%%%%%%%%%%%%%%%%%%%%%
%%%%%%%%%%%%%%%%%%%%%%%%

In this section we construct a sequence of curves that satisfies the condition $\mathcal P$ of Definition~\ref{def : sequence of curves}.
This construction is a generalization of the one in \cite{nonuniqueerg} to spheres with more punctures. 
Fix a sequence of positive integers $\mathcal E = \{e_k\}_{k=0}^\infty$.

Let $p\geq 5$ be an odd integer and $S = S_{0,p}$ be a sphere with $p$ punctures.  We visualize $S$ as the double of a regular $p$-gon (with vertices removed), admitting an order $p$ {\em rotational symmetry}, as in Figure~\ref{fig : 07}.  Let $\rho\colon S \to S$ be the counterclockwise rotation by angle $4\pi/p$.  Set $m = \frac{p-1}2$.

Next, let $\gamma_0$ be a curve obtained by doubling an arc connecting two sides of the polygon adjacent to a common side.  Then $\{\rho^j(\gamma_0)\}_{j=0}^{p-1}$ is a set of $p$ curves that pairwise intersection $0$ or $2$ times; see Figure~\ref{fig : 07}.   We let $\alpha = \rho^m(\gamma_0)$, and recall that $D_\alpha$ denotes the positive Dehn twist about the curve $\alpha$.  For $k \geq 1$, set
\begin{eqnarray}\label{eq : phi}
 \phi_k &=& D_\alpha^{e_{k+m-1}} \rho,
\;\text{and}\\
 \Phi_k &=& \phi_1 \phi_2 \cdots \phi_k,\nonumber
 \end{eqnarray}
(in particular, $\Phi_0 = id$).

Define a sequence of curves $\Gamma = \Gamma(\mathcal E) = \{\gamma_k\}_{k=0}^\infty$, starting with $\gamma_0$, by the formula
\begin{equation}\label{eq : gamma}
 \gamma_k = \Phi_k(\gamma_0).
 \end{equation}
Since a twist about $\alpha$ has no effect on a curve disjoint from it, for $0 \leq j \leq 2m-1$, 
\begin{equation} \label{eqn : alt def gamma} \gamma_k = \Phi_{k-j}(\gamma_j)= \Phi_{k-m}(\alpha),
\end{equation}
for all $k\geq m$.
%The choices involved ensure that $\{\gamma_k\}_{k=0}^\infty$ satisfies the conditions of Definition~\ref{def : sequence of curves}.
See Figure~\ref{F : 7-in-a-row} for a picture illustrating $2m+1$ consecutive curves.

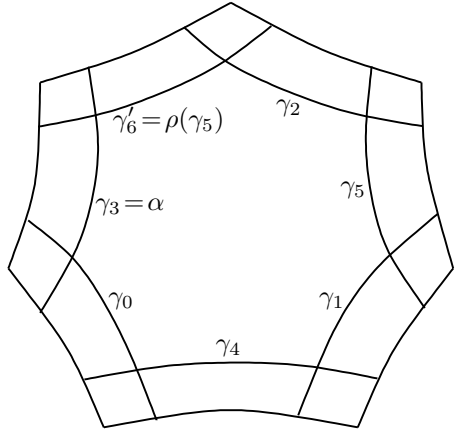
\begin{figure} [htb]
\labellist
\small\hair 2pt
\pinlabel $\gamma_0$ [b] at 27 27
\pinlabel $\gamma_1$ [b] at 76 27
\pinlabel $\gamma_2$ [b] at 66 72
\pinlabel $\gamma_3\!=\!\alpha$ [b] at 29 49
\pinlabel $\gamma_4$ [b] at 52 16
\pinlabel $\gamma_5$ [b] at 81 53
\pinlabel $\gamma_6'\!=\!\rho(\gamma_5)$ [b] at 38 68
\endlabellist
\begin{center}
\includegraphics[width=6cm]{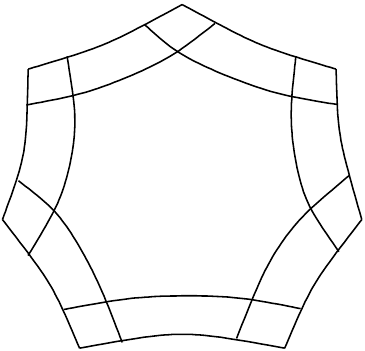} \caption{$S_{0,7}$ as a double of a $7$--gon.  The curves $\gamma_0,\gamma_1,\gamma_2,\gamma_3=\alpha,\gamma_4,\gamma_5$ and $\gamma'_6=\rho(\gamma_5)$ are shown.}
\label{fig : 07}
\end{center}
\end{figure}

\begin{prop}\label{prop : sequence on S_0,p}
The sequence $\Gamma(\mathcal E) = \{\gamma_k\}_{k=0}^\infty$ satisfies condition $\mathcal P(\mathcal E)$ in Definition~\ref{def : sequence of curves} for $m = \frac{p-1}{2}$ and $b' = b = 2$.
\end{prop}
\begin{figure} [htb]
\labellist
\small\hair 2pt
\pinlabel $\gamma_0$ [b] at 24 87
\pinlabel $\gamma_1$ [b] at 113 87
\pinlabel $\gamma_2$ [b] at 183 110
\pinlabel $\gamma_3=\alpha$ [b] at 249 98
\pinlabel $\gamma_4$ [b] at 52 13
\pinlabel $\gamma_5$ [b] at 151 30
\pinlabel $\gamma_6^{(k)}$ [b] at 233 30
\pinlabel $\stackrel{2e_k}{}$ [b] at 200 3
\endlabellist
\begin{center}
\includegraphics[width=8cm]{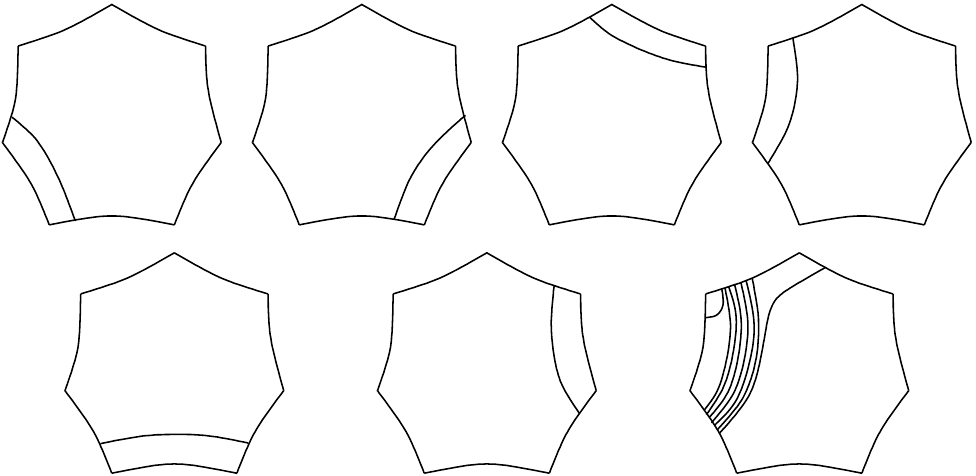} \caption{For $S_{0,7}$ and any $k \geq 3$, applying $\Phi_{k-3}$ to any seven consecutive curves in the sequence, $\gamma_{k-3},\ldots,\gamma_{k+3}$, gives $\gamma_0,\ldots,\gamma_5,\gamma_6^{(k)} = \Phi_{k-3}(\gamma_{k+3})$ as shown here.}
\label{F : 7-in-a-row}
\end{center}
\end{figure}
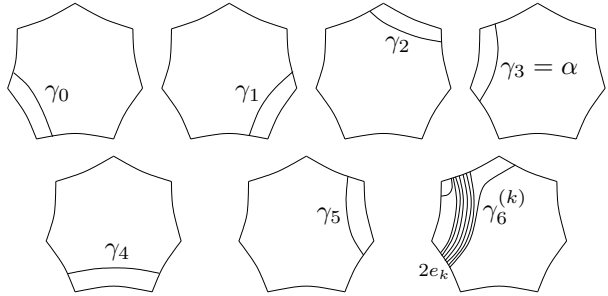

\begin{proof} The proof boils down to showing that, after applying an appropriate homeomorphism, any $2m+1 = p$ consecutive curves differ from $\gamma_0,\ldots,\gamma_{2m}$ only in the amount of relative twisting of $\gamma_0$ and $\gamma_{2m}$ around $\gamma_m$; see Figure~\ref{F : 7-in-a-row} and compare with the construction from \cite{nonuniqueerg}.  We now explain this in more detail.

First, observe that for $j = 1,\ldots,p-2 = 2m-1$, $\I(\alpha,\rho^j(\gamma_0)) = 0$.  Thus
\[ \gamma_j = \rho^j(\gamma_0),\]
for $j = 0,\ldots,2m-1$.
By construction, any two of these curves intersect $0$ or $2$ times, while the first $m$ are pairwise disjoint.  Furthermore, the entire set of $2m$ curves fills $S$; see Figure~\ref{fig : 07}.

Next, for any $k \geq m$, applying $\Phi_{k-m}^{-1}$ to the $2m-1$ consecutive curves $\{\gamma_{k-m},\ldots,\gamma_{k+m-1}\}$, (\ref{eqn : alt def gamma}) implies
\begin{equation}
\label{eqn : alt gamma def} \Phi_{k-m}^{-1}(\gamma_{k-m+j}) = \gamma_j
\end{equation}
for each $j = 0,\ldots,2m-1$.  Since $k$ was arbitrary, it follows that any $m$ consecutive curves are pairwise disjoint, and any $2m$ consecutive fill $S$.  Thus, conditions (i) and (ii) of $\mathcal P$ are satisfied.

For part (iii), let $\gamma_{k+m}' = \Phi_{k-m}(\rho(\gamma_{2m-1})) = \Phi_{k-m}(\rho^{2m}(\gamma_0))$.    Then, for $j = 0,\ldots,2m-1$, we may apply $\Phi_{k-m}^{-1}$, and we have
\[ \I(\gamma_{k+m}',\gamma_{k-m+j}) = \I(\rho^{2m}(\gamma_0),\gamma_j) = \I(\rho^{2m}(\gamma_0),\rho^j(\gamma_0)) = \left\{ \begin{array}{ll} 2 & \mbox{for $j=m$,}\\
&\mbox{and $m-1$}\\
0 & \mbox{otherwise,} \end{array} \right. \]
which implies the intersection number requirement for (iii), with $b' = b = 2$.

Finally, applying $\Phi_{k-m}$ to $\gamma_{k+m}$ we get
\begin{eqnarray*}
\Phi_{k-m}^{-1}(\gamma_{k+m})&  = & \phi_{k-m+1}\cdots \phi_{k+m}(\gamma_0) = \phi_{k-m+1}(\rho^{2m-1}(\gamma_0))\\
& = & D_\alpha^{e_{(k-m+1)+m-1}}\rho(\rho^{2m-1}(\gamma_0)) = D_{\gamma_m}^{e_k}(\rho^{p-1}(\gamma_0)),
\end{eqnarray*}
where we have used the fact that $\alpha = \rho^m(\gamma_0) = \gamma_m$.  Therefore,
\begin{eqnarray*} \gamma_{k+m} & = &\Phi_{k-m}(D_{\gamma_m}^{e_k}(\rho^{p-1}(\gamma_0))\\
& = &\Phi_{k-m}D_{\gamma_m}^{e_k}\Phi_{k-m}^{-1}\Phi_{k-m}\rho^{p-1}(\gamma_0)\\
& = &D_{\Phi_{k-m}(\gamma_m)}^{e_k}(\gamma_{k+m}') = D_{\gamma_k}^{e_k}(\gamma_{k+m}').
\end{eqnarray*}
Therefore, part (iii) from $\mathcal P$ is also satisfied.
\end{proof}

\begin{cor} \label{cor : sequence to ending lamination}  If $\mathcal E = \{e_k\}_{k=0}^\infty$ satisfies $e_{k+1} \geq a e_k$ for all $k$ and for an $a > 1$ sufficiently large, then the conclusions of Theorem~\ref{thm : subsurface coeff estimate}, Theorem~\ref{thm : intersection} and Theorem~\ref{thm : nue} hold for the sequence $\Gamma(\mathcal E) = \{\gamma_k\}_{k=0}^\infty$.  In particular, the sequence determines a minimal, filling non-uniquely ergodic lamination $\nu$.
\end{cor}

%%%%%%%%%%%%%%%%%%%%%%%%
%%%%%%%%%%%%%%%%%%%%%%%%
\section{Limits of closed geodesics}\label{sec : closed geodesics}
%%%%%%%%%%%%%%%%%%%%%%%%
%%%%%%%%%%%%%%%%%%%%%%%%

Let $S=S_{0,p}$ be the $p$--punctured sphere where $p\geq 5$ is an odd integer. Let $\alpha,\rho$ then be as in \S\ref{sec : sequence of curves on S0p}, $e$ be a positive integer, and ${f}_e = D_\alpha^e\rho$.
To relate this to the previous section, note that for any fixed $e$, the sequence of mapping classes $\{\phi_k\}_{k=0}^\infty$ obtained from the constant sequence $\mathcal E = \{e\}_{k=0}^\infty$, is constant; $\phi_k = {f}_e$ for all $k$.  Consequently, the sequence of curves $\Gamma(\mathcal E)$ is obtained by iteration: $\Gamma(\mathcal E) = \{{f}_e^k(\alpha)\}_k$ (after a shift of indices).  

We assume in the following that $e > E$ from Theorem~\ref{thm : subsurface coeff estimate}.  Then by Proposition~\ref{prop : sequence on S_0,p} and Theorem~\ref{thm : subsurface coeff estimate}, $k \mapsto {f}_e^k(\alpha)$ is a ${f}_e$--invariant quasi-geodesic in the curve complex, and hence $f_e$ is pseudo-Anosov. By \cite[Theorem 4.1]{bdrycc} the sequence of curves $\{f_e^k(\alpha)\}_{k=0}^\infty$ determines a projective measured lamination $[\bar\nu_e^+]$ and $\{f_e^k(\alpha)\}_{k=0}^{-\infty}$ determines a projective measured lamination $[\bar\nu_e^-]$. 

%\bm{I added a map for the section here, I also changed the word behavior to picture}
A key ingredient in our construction of a Weil-Petersson geodesic ray in \S\ref{sec : NUE case} will be a very precise understanding of the limiting picture of the axes $g_e$ of the pseudo-Anosov mapping classes ${f}_e$, as $e$ tends to infinity. The main results of this section are Proposition \ref{prop : converge to omega} in which we describe a biinfinite piecewise geodesic in $\overline{\Teich(S)}$ which approximate the geodesics $g_e$ in the Hausdorff topology and gives us the necessary limiting picture for $g_e$ as $e\to\infty$. 
\medskip

Our analysis of the axes $g_e$ of ${f}_e$ begins with an analysis of the action of $\rho$ on $\Teich(S)$ and certain strata in the Weil-Petersson completion.
Observe that the quotient of $S$ by $\langle \rho \rangle$ is a sphere with one puncture and two cone points.  A fixed point of $\rho$ in $\Teich(S)$ is a $\rho$--invariant conformal structure on $S$, or equivalently, a conformal structure obtained by pulling back a conformal structure on the quotient $S/\langle \rho \rangle$.
Since the sphere with three marked points is rigid, there is a unique such conformal structure, and hence exactly one fixed point $Z \in \Teich(S)$ for the action of $\rho$.

\begin{prop}\label{prop : minimum of f}
For the stratum $\cS(v)$ defined by a multicurve $v$, there exists a point $X_0 \in \overline{\cS(v)}$ so that 
\[
d_{\WP}(X_0, \rho(X_0)) = \inf_{Y\in \cS(v)} d_{\WP}(Y, \rho(Y)). 
\]
\end{prop}

\begin{remark}
Note that unless $\cS(v)$ is a point (i.e.~$v$ is a pants decomposition), $\overline{\cS(v)}$ is not compact.
\end{remark}

Now we define the function
\[
F:\overline{\Teich(S)}\to\mathbb{R}^{\geq 0},
\]
by $F(X)=d_{\WP}(X,\rho(X))$. The proposition is then equivalent to showing that the restriction of $F$ to the closure $\overline{\cS(v)}$ attains a minimum value.
We begin with a lemma.

\begin{lem} \label{lem : f is proper} 
The function $F:\overline{\Teich(S)}\to\mathbb{R}^{\geq 0}$ is convex, $2$--Lipschitz, and for any $R >0$, $F^{-1}([0,R])$ is a bounded set.
\end{lem}
%\bm{I slightly modified and expanded the proof}
\begin{proof}
Since the completion of the Weil-Petersson metric is $\CAT(0)$, the distance function on $\overline{\Teich(S)}$ is convex, and hence so is $F$.  The triangle inequality proves that $F$ is $2$--Lipschitz, since
\begin{eqnarray*}
|F(Y)-F(X)| & = & \big|d_{\WP}(Y,\rho(Y)) - d_{\WP}(X,\rho(X))\big| \\
&\leq& \big|d_{\WP}(X,Y)+d_{\WP}(X,\rho(Y)) - d_{\WP}(X,\rho(X))\big|\\
&\leq & d_{\WP}(X,Y) + d_{\WP}(\rho(X),\rho(Y)) \\
&=& 2 d_{\WP}(X,Y).
\end{eqnarray*}

Let $Z \in \Teich(S)$ be the fixed point of the action of $\rho$ on $\Teich(S)$ and suppose that $R_0 > 0$ is sufficiently small so that $B_{R_0}(Z)$, the closed ball of radius $R_0$ in $\overline{\Teich(S)}$ about $Z$, is contained in $\Teich(S)$, and thus is compact.  Let $R_1 > 0$ be the minimum value of $F$ on $\partial B_{R_0}(Z)$.

For any $Y \in \overline{\Teich(S)} \setminus B_{R_0}(Z)$, let $Y_0$ be the unique point of intersection of the geodesic from $Z$ to $Y$ with the sphere $\partial B_{R_0}(Z)$.  Then it follows that $d_{\WP}(Y,Z) = R_0 + d_{\WP}(Y,Y_0)$, and so convexity of $F$ implies
\[
F(Y_0)\leq \frac{R_0}{d_{\WP}(Y,Z)} F(Y)+\frac{d_{\WP}(Y,Y_0)}{d_{\WP}(Y,Z)} F(Z)
\]
But then since $F(Z)=0$ and $F(Y_0)\geq R_1$ we have
\[ F(Y) \geq \frac{d_{WP}(Z,Y)}{R_0} F(Y_0) \geq \frac{d_{WP}(Z,Y)}{R_0} R_1.\]
Rearranging the above inequality, we have
\[ d_{WP}(Z,Y) \leq \frac{R_0}{R_1} F(Y) \]
and hence if $R > R_0$ and $F(Y) \leq R$, then we have $d_{\WP}(Z,Y) \leq \frac{R_0R}{R_1}$.  That is, $F^{-1}([0,R]) \subset B_{R_0R/R_1}(Z)$, as required.
\end{proof}

\begin{proof}[Proof of Proposition~\ref{prop : minimum of f}.]
Any stratum in $\overline{\cS(v)}$ has the form $\cS(v')$ for a multicurve $v'$ containing $v$.
Observe that the infimum of the function $F$ on any stratum $\cS(v')$ in $\overline{\cS(v)}$ is no less than the infimum of $F$ on $\cS(v)$.  Let $\cS(v')$ in $\overline{\cS(v)}$ be a stratum in the closure having minimal dimension, so that the infimum of $F$ on $\cS(v')$ is equal to the infimum on $\cS(v)$.  It suffices to show that the infimum of $F$ on $\cS(v')$ is realized on $\cS(v')$. 

Let $\{X_n\}_{n=1}^\infty \subset \cS(v')$ be an {\em infimizing sequence} for $F$ on $\cS(v')$; that is 
\begin{equation}
\lim_{n\to\infty} F(X_n)=\inf_{X\in\cS(v')} F(X).
\end{equation} 
Let $R <\infty$ be such that $F(X_n) \leq R$ for all $n \geq 1$.
Lemma \ref{lem : f is proper} then implies that there exists $D > 0$ such that $d_{\WP}(Z,X_n) \leq D$ for all $n \geq 1$.

By the triangle inequality, the lengths of the geodesic segments $[X_1,X_n]$ are bounded by $2D$. Let $S_1,\ldots, S_m$ be the connected component of $S\backslash v'$. Then $\cS(v')$ is isometric to $\prod_{j=1}^m\Teich(S_j)$ with the product of WP metrics on each factor. Let $\zeta_n:[0,T_n]\to\prod_{j=1}^m\Teich(S_j)$ be the parametrization of $[X_1,X_n]$ by arc length. Let 
\[\pr_j:\prod_{j=1}^m \Teich(S_j)\to \Teich(S_j)\]
 be the projection to the $j$-the factor and let $\zeta^j_n:[0,T^j_n]\to \Teich(S_j)$ be parametrization by arc length of $\pr_j\circ \zeta_n$. Note that $T^j_n\leq 2D$ for $j=1,\ldots,m$. So for a fixed $j$, trimming the intervals and reparametrization we get a sequence of geodesic segments $\zeta^j_n:[0,T^j]\to \Teich(S_j)$ of equal length. We may then apply Theorem \ref{thm : geodesic limit} (Geodesic Limit Theorem) to the sequence of geodesic segments $\zeta^j_n:[0,T^j]\to\Teich(S_j)$.  Let the multicurves $\sigma^j_i$, $i=0,\ldots,k_j+1$, the multicurve $\hat{\tau}^j$, the partition $t^j_0<\ldots<t^j_{k_j+1}$ and the piecewise geodesic $\hat{\zeta}^j$ be as in the theorem. Also let the elements of mapping class group $\psi^j_n$ and $\varphi^j_{l,n},\; l=1,\ldots,k_j$ be as in the theorem. Note that by the theorem when $k_j\geq 1$ we have that $\hat{\zeta}^j(t^j_1)\in\cS(\sigma^j_1)$ and $\lim_{n\to\infty}\varphi^j_{1,n}(\zeta^j_n(t^j_1))=\hat{\zeta}^j(t^j_1)$.
 
Since the geodesics $\zeta^j_n$ have a common starting point $\pr_j(X_1)$, it follows that $\psi^j_n$ is the identity map for all $n$.
Hence, if $k_j=0$, then after possibly passing to a subsequence the points $\pr_j(X_n)$ converge. 
\medskip

First suppose that $k_j=0$ for all $j=1,\ldots,m$, then after possibly passing to a subsequence all sequences $\pr_j(X_n)$ converge as $n\to\infty$. As a result the points $X_n$ converge and we are done.
 
Now we suppose that $k_j\geq 1$ for some $j$, let $\beta\in\sigma^j_1$, and we derive a contradiction. Note that $\pr_j(X_1)$ and $\pr_j(X_n)$ are in $\Teich(S_j)$. Claim 4.9 in the proof of Theorem 4.6 in \cite{wpbehavior} tells us that for Bers markings $\mu(\pr_j(X_1))$ and $\mu(\pr_j(X_n))$ and curves $\beta_n=(\varphi_{1,n}^j)^{-1}(\beta)$, 
\[d_{\beta_n}\Big(\mu(\pr_j(X_1)),\mu(\pr_j(X_n))\Big)\to \infty.\]
as $n\to\infty$. 

Now recall that $\varphi_{1,n}^j=\mathcal{T}_{1,n}^j\circ \psi^j_n$, also that $\mathcal{T}_{1,n}^j$ is the composition of a power of the Dehn twist about the curve $\beta_n$ and powers of Dehn twists about curves disjoint from $\beta_n$. Moreover, as we saw above $\psi_n$ is identity. Thus $\beta_n\equiv\beta$ for all $n$. Therefore the above limit becomes 
\[d_{\beta}\Big(\mu(\pr_j(X_1)),\mu(\pr_j(X_n))\Big) \to \infty\]
 as $n\to\infty$. 
 We may then choose a sequence $\{n_k\}_{k=1}^\infty$ so that 
\begin{equation}\label{eq : dbXn}
d_{\beta}\Big(\mu(\pr_j(X_{n_k})),\mu(\pr_j(X_{n_{k+1}}))\Big)\to \infty
\end{equation}
 as $k\to\infty$. 

 \begin{claim}
There exists a sequence of points $\{Y_{n_k}\}_k$ on the geodesic segments $[X_{n_k},X_{n_{k+1}}]$ with the property that the distance between $Y_{n_k}$ and $\cS(v' \cup \beta)$ goes to $0$.
 \end{claim}
\begin{proof}
It suffices to show that there is a sequence of points $Y_{n_k}$ on $[X_{n_k},X_{n_{k+1}}]$ so that the distance between $\pr_j(Y_{n_k})$ and $\cS(\beta) \subset \overline{\Teich(S_j)}$ goes to $0$.  If the distance between $\pr_j(X_{n_k})$ and $\cS(\beta)$ goes to $0$, the sequence $Y_{n_k} = X_{n_k}$ is the desired sequence.  
Otherwise, there is a lower bound for the distance between $\pr_j(X_{n_k})$ and $\cS(\beta)$.  Moreover, by Corollary~\ref{cor : dist to stratum} we have 
\[d_{\WP}(\pr_j(X_{n_k}),\cS(\beta))\leq \sqrt{2\pi \ell_{\beta}(\pr_j(X_{n_k}))}.\]
 Thus we obtain a lower bound for $\ell_{\beta}(\pr_j(X_{n_k}))$. Appealing to Theorem \ref{thm  : tw-sh}, the annular coefficient limit (\ref{eq : dbXn}) provides a point $Y^j_{n_k}$ on $[\pr_j(X_{n_k}),\pr_j(X_{n_{k+1}})]$ so that $\ell_{\beta}(Y^j_{n_k})\to 0$, and hence again by  Corollary~\ref{cor : dist to stratum} the distance between $Y^j_{n_k}$ and $\cS(\beta)$ goes to $0$. Now the points $Y_{n_k}$ on $[X_{n_k},X_{n_{k+1}}]$ with $\pr_j(Y_{n_k})=Y^j_{n_k}$ are the desired points.
\end{proof}

It follows from the above claim and the convexity of the function $F$ that
%\bm{I changed min to max}
\begin{equation}\label{eq : fY fX}
F(Y_{n_k})\leq \max\{F(X_{n_k}),F(X_{n_{k+1}})\}.
\end{equation}
Therefore, $\{Y_{n_k}\}$ is also an infimizing sequence for the function $F$ on $\cS(v)$.
Let $Y'_k$ be the closest point to $Y_{n_k}$ in $\cS(v' \cup \beta)$.  Since the distance of the points $Y_{n_k}$ and $\cS(v' \cup \beta)$ goes to $0$ we have that
\[d_{\WP}(Y_{n_k},Y'_{k})\to 0\] 
and therefore $F(Y_{n_k})$ and $F(Y'_n)$ have the same limit since $F$ is $2$--Lipschitz.  Therefore $\{Y'_{k}\}_{k}$ is a infimizing sequence for the function $F$ in the stratum $\cS(v' \cup \beta)$, but this stratum has dimension less than that of $\cS(v')$. This contradiction finishes the proof of the proposition.
\end{proof}

%\bm{I added this line and transported the following paragraph from page 19 and restated Proposition 4.5. I think it is better to have the description of $g^{\omega}_{\mathcal E}$ at the beginning} 

Now we can describe the biinfinite piecewise geodesics $g^{\omega}_e\subset \overline{\Teich(S)}$ which approximate the geodesics $g_e$, the axes of the pseudo-Anosov mapping classes $f_e$ as follows.
%Recall that $f_e= D_\alpha^e\rho$, where $\alpha$ is a fixed curve on $S=S_{0,p}$ and $\rho$ is a finite order element of $\Mod(S)$ (see \S\ref{sec : sequence of curves on S0p}). 
First, appealing to Proposition~\ref{prop : minimum of f}, let $X_0 \in \overline{\cS(\rho^{-1}(\alpha))}$ be a point where the function $F(X) = d_{WP}(X,\rho(X))$ is minimized in the closure of the stratum $\cS(\rho^{-1}(\alpha))$.  As already observed, on $\overline{\cS(\rho^{-1}(\alpha))}$, we have ${f}_e = D_\alpha^e \rho = \rho$ since $D_\alpha^e$ acts trivially on $\rho(\overline{\cS(\rho^{-1}(\alpha))}) = \overline{\cS(\alpha)}$.  Consequently, ${f}_e(X_0) = \rho(X_0)$, and we may concatenate the geodesic segment $\omega = [X_0,\rho(X_0)]$ with its ${f}_e$--translates to produce an ${f}_e$--invariant, biinfinite piecewise geodesic in $\overline{\Teich(S)}$:
\begin{equation} \label{eq:g-omega-e}
g_e^\omega = \cdots \cup {f}_e^{-2}(\omega) \cup {f}_e^{-1}(\omega) \cup \omega \cup {f}_e(\omega) \cup {f}_e^2 (\omega) \cup \cdots 
\end{equation}

\begin{prop} \label{prop : converge to omega}
The path $g_e^{\omega}$ is a biinfinite piecewise geodesic that fellow travels $g_e$, and the Hausdorff distance between $g_e^\omega$ and $g_e$ tends to $0$ as $e \to \infty$.
\end{prop}

For the proof of the proposition we need the following theorem which is a {\em characterization of the short curves} along the geodesic $g_e$. 
In the following let $E>0$ be the constant from Theorem~\ref{thm : subsurface coeff estimate}.

\begin{thm} \label{thm : short curves along ge} There exists $\epsilon > 0$ so that for all $e > E$ and every point of $g_e$, at most one curve on $S$ has length less than $\epsilon$, and such a curve is in the set $\{{f}_e^k(\alpha)\}_{k \in \mathbb Z}$ (with $\alpha$ as in \S\ref{sec : sequence of curves on S0p}). Moreover, let $t_e$ be the translation length of $f_e$, then after reparametrization of $g_e$ we have that the minimal length of the curve ${f}_e^k(\alpha), \; k \in \mathbb Z$ along $g_e$ is realized at $kt_e$ and tends to zero as $e \to \infty$. 
\end{thm}
\begin{proof}

Let $\nu_e^{\pm}$, as before, be the laminations determined by the sequences of curves $\{f_e^k(\alpha)\}_{k=0}^{\pm\infty}$. There is a uniform bound for all subsurface coefficients of the pairs $(\nu_e^-,\nu_e^+)$ except those of $\{f_e^k(\alpha)\}_{k\in\mathbb Z}$. This follows from the fact that in Theorem \ref{thm : subsurface coeff estimate} the upper bound $R$ depends only on the parameters from Definition \ref{def : sequence of curves} and the initial marking $\mu$ which is the same for all ${f}_e$. 

Similarly we have 
 \begin{equation}\label{eq : tw fek(alpha)}d_{f_e^k(\alpha)}(\nu_e^-,\nu_e^+)\asya e\end{equation}
  for all $k\in\mathbb Z$, where the additive error is independent of $e$.

Let $\varrho_e:[-\infty,+\infty]\to P(S)$ be a hierarchy path between the pair $(\nu_e^-,\nu_e^+)$ (see \cite{mm2}). Since the pair has non-annular $R$--bounded combinatorics $\varrho_e$ is stable in $P(S)$ \cite[Theorem 4.3]{bmm2}\cite[Theorem 5.13]{wpbehavior}. Therefore, $\varrho_e$ and $Q(g_e)$ the image of $g_e$ under Brock's quasi isometry (\ref{eq : Brock qi}) $D$--fellow travel, where the constant $D\geq 0$ depends only on $R$.

\begin{lem}\label{lem : ge recur}
There is an $\ep_2>0$, so that for all $e>E$, $g_e$ visits the $\ep_2$--thick part of $\Teich(S)$ infinitely often in both forward and backward times. 
\end{lem}

\begin{proof}
Let $\mu(g_e(0))$ be a Bers marking at $g_e(0)$, and let $i$ be so that $\varrho_e(i)$ is within distance $D$ of $Q(g_e(0))$. Then all non-annular subsurface coefficients of the pair $(\varrho_e(i),Q(g_e(0)))$ are bounded by $\max\{KD+KC,A\}$ by the distance formula (\ref{eq : MM dist}) for a choice of threshold $A$. Moreover, by Theorem \ref{thm : no back tracking}  all non-annular subsurface coefficients of the pair $(\varrho_e(i),\nu_e^+)$ are bounded by an enlargement of $R$. Combining the bounds with the triangle inequality in the curve complex of each subsurface then implies that $(\mu(g_e(0)),\nu_e^+)$, the end invariant of $g_e|_{[0,\infty)}$, has non-annular bounded combinatorics, independent of $e$. Theorem \ref{thm : bd comb recurrence}, then guarantees that for an $\ep_2>0$, independent of $e$, the geodesic ray $g_e|_{[0,\infty)}$ visits the $\ep_2$--thick part of $\Teich(S)$ infinitely often. The proof of that the geodesic ray $g_e|_{[0,-\infty)}$ visits the $\ep_2$--thick part of $\Teich(S)$ infinitely often is similar.
\end{proof}

Now we prove the following:
\begin{lem}\label{lem : lb}
 There exists $\ep_1>0$, depending only on $R$, so that for all $e> E$ the length of each curve $\gamma\notin\{f_e^k(\alpha)\}_k$ is bounded below by $\ep_1$ along $g_e$.
 \end{lem}
 \begin{proof}

Suppose that for a $t\in\mathbb{R}$ the length of $\gamma$ at $g_e(t)$ is less than the Bers constant. Then, $\gamma\in Q(g_e(t))$ and thus $\gamma$ is isolated at some $i$ along $\varrho_e$; for the discussion about isolated annular subsurfaces see $\S$\ref{subsec : isolated annular subsurf}. By Lemma \ref{lem : ann coeff comp}, there are constants ${\bar w}, b$ depending only on $R$ and a constant $L$ such that for any $s>{\bar w}$ and $s'\asya_b s$,
$$\min\{\ell_{\gamma}(g_e(t-s')),\ell_{\gamma}(g_e(t+s'))\}\geq L.$$

 Fix $s,s'$ as above and fix $u< s'$. Let $J=[t-s'+u,t+s'-u]$.  
 Then, Theorem \ref{thm : shtw} applies to the geodesic segment $g_e|_{[t+s',t-s']}$ and implies that for any integer $N\geq 1$,  there is an $\epsilon\in(0,L/2)$ so that 
 \begin{equation}\label{eq : shtw}
 \text{if}\; \inf_{r\in J}\ell_{\gamma}(g_e(r))<\epsilon,\text{then}\; d_{\gamma}(\mu(g_e(t-s')),\mu(g_e(t+s')))>N,
 \end{equation}
where $\mu(\cdot)$ denotes a Bers marking at the given point. 

 According to Lemma \ref{lem : ann coeff comp} there is a constant $B>0$ depending only on $R$ such that
\begin{equation}\label{eq : acc}d_{\gamma}(\mu(g_e(t-s')),\mu(g_e(t+s'))) \asya_{B} d_{\gamma}(\varrho_e(i-s),\varrho_e(i+s)).\end{equation}

Further, suppose that $\gamma$ is not in the set $\{{f}_e^k(\alpha)\}_{k \in \mathbb Z}$. Then the upper bound for $d_\gamma(\nu_e^-,\nu_e^+)$
and Theorem \ref{thm : no back tracking} for the parameters $-\infty,i-s,i+s,\infty$ of $\varrho_e$ give us an upper bound for the subsurface coefficient 
\[d_{\gamma}(\varrho_e(i-s),\varrho_e(i+s))\]
 depending only on $R$. So by (\ref{eq : acc}) we get an upper bound for 
 \[d_{\gamma}(\mu(g_e(t-s')),\mu(g_e(t+s'))\]
  depending only on $R$. On the other hand, since $t\in J$ by (\ref{eq : shtw}) if $\ell_\gamma(g(t))$ gets arbitrary small, then $d_{\gamma}(\mu(g_e(t-s')),\mu(g_e(t+s')))$
 would become arbitrary large, which contradicts the upper bound we just obtained. Therefore, there is a lower bound $\ep_1>0$ for the length of $\gamma$ at time $t$ which depends only on $R$. Since $t$ was arbitrary the proof of the lemma is complete.
\end{proof}

 The length of each one of the curves in the set $\{{f}_e^k(\alpha)\}_{k \in \mathbb Z}$ is strictly convex along $g_e$ (\cite{wolb}), and so has a unique minimum. The unique minimum for ${f}_e^k(\alpha)$ occurs at the ${f}_e^k$--image of the point where $\alpha$ is minimized.  Thus, we can parameterize $g_e$ by arc length so that for $t_e$ the WP translation length of ${f}_e $, the length of the curve ${f}_e^k(\alpha)$ is minimized at $g_e(kt_e)$.

By Lemma \ref{lem : ge recur} there is an $\ep_2>0$ so that for all $e>E$, $g_e$ visits the $\ep_2$--thick part infinitely often in both forward and backward times. Let $t_e' \in (0,t_e)$ be a time for which $g_e(t_e')$ is in the $\ep_2$--thick part.  But then $g_e(kt_e + t_e')$ is in the thick part for all $k$.  In particular, by convexity of the length of $\alpha$, it follows that outside the interval $(-t_e+t_e',t_e')$, $2\ep_2$ is a uniform lower bound for the length of $\alpha$.  Likewise, the length of ${f}_e^k(\alpha)$ is uniformly bounded below by $2\ep_2$ outside the interval $((k-1)t_e + t_e',k t_e + t_e')$.  Consequently, for $k \neq k'$, the curves ${f}_e^k(\alpha)$ and ${f}_e^{k'}(\alpha)$ cannot simultaneously have length less than $2\ep_2$. 

As we saw in Lemma \ref{lem : lb} the only curves which can get shorter than $\ep_1$ along $g_e$ are $\{f_e^k(\alpha)\}_k$. Moreover, since we saw above that two of these curves cannot get shorter than $\ep_2$ at the same time, the first statement of the theorem holds for $\ep=\min\{\ep_1,2\ep_2\}$.
\medskip

Let the laminations $\nu_e^{\pm}$ be as before, and let $\varrho_e$ be a hierarchy path between $\nu_e^-$ and $\nu_e^+$. Recall that $\varrho_e$ is stable and that $\varrho_e$ and $Q(g)$, $D$--fellow travel for a $D$ that depends only on $R$. 

Note that by (\ref{eq : tw fek(alpha)}) each curve $f_e^k(\alpha)$ ($k\in\mathbb Z$) for $e$ sufficiently large has a big enough subsurface coefficient that $f_e^k(\alpha)$ is in $\varrho_e(i)$ for an $i$ in the domain of $\varrho_e$ by \cite[Lemma 6.2 (Large link)]{mm2}. Thus $f_e^k(\alpha)$ is isolated at $i$ along $\varrho_e$ (see $\S$\ref{subsec : isolated annular subsurf}). Let $t$ be a time so that $Q(g_e(t))$ is within distance $D$ of $\varrho_e(i)$. Then for constants ${\bar w}, b,B$ from Lemma \ref{lem : ann coeff comp} and any $s\geq {\bar w}$ and $s'\asya_b s$ we have that 
\[d_{\gamma}(\mu(g_e(t-s')),\mu(g_e(t+s'))) \asya_{B} d_{\gamma}(\varrho_e(i-s),\varrho_e(i+s)).\] 
Thus the bound $d_{f_e^k(\alpha)}(\nu_e^-,\nu_e^+)\asya e$ and Theorem \ref{thm : no back tracking} for the parameters $-\infty,i-s,i+s,\infty$ of $\varrho_e$ imply that 
\[d_{\gamma}(\mu(g_e(t-s')),\mu(g_e(t+s')))\asya e.\]
 Moreover, for the constant $L>0$ form Lemma \ref{lem : ann coeff comp} we have that
 \[\min\{\ell_{\gamma}(g_e(t-s')),\ell_{\gamma}(g_e(t+s'))\}\geq L.\] 
 
 Now fix $s,s'$, then Theorem \ref{thm : tw-sh} applies to geodesic segment $g_e|_{[t-s',t+s']}$ and implies that 
 \[\inf_{r\in[t-s',t+s']} \ell_{{f}_e^k(\alpha)}(g_e(t))\to 0\]
  as $e\to\infty$. But the minimal length of $f_e^k(\alpha)$ is realized at $kt_e$ so
\[\lim_{e\to\infty}\ell_{{f}_e^k(\alpha)}(g_e(kt_e))= 0.\]
This completes the proof of the second statement of the theorem.
\end{proof}

We continue to use $t_e > 0$ to denote the WP translation length of ${f}_e$ and assume the geodesic $g_e$ is parameterized as in the proof of the theorem above.  
Then, in particular the minimal length of $f_e^k(\alpha)$ along $g_e$ is realized at time $kt_e$ and $\ell_{f^k_e(\alpha)}(g_e(kt_e))\to 0$ as $e\to\infty$.
Likewise, $\{(k-1)t_e + t_e'\}_{k \in \mathbb Z}$ denotes times when $g_e$ intersects the fixed thick part of $\Teich(S)$. Also, note that the minimum of the length of $f_e^k(\alpha)$ is realized at $kt_e$ ($k\in \mathbb Z$) and $\lim_{e\to\infty}\ell_{f_e^k(\alpha)}(kt_e)=0$.

To prove Proposition~\ref{prop : converge to omega} we also need the following lemma about the limit of translation length of $f_e$. 

\begin{lem} \label{lem : lim te}
 The translation distance $t_e$ of ${f}_e$ limits to $|\omega|$ the length of $\omega$; that is,
\[ \lim_{e \to \infty} t_e = d_{\WP}(X_0,\rho(X_0)) = |\omega|, \]
where $X_0\in \overline{\cS(\rho^{-1}(\alpha))}$, as before, is the point where $ d_{\WP}(X_0,\rho(X_0))=\inf_{X\in\cS(\rho^{-1}(\alpha))}d_{\WP}(X,\rho(X))$.
\end{lem}
\begin{proof} %\bm{I slightly changed and expanded the proof}
Let $Y_e \in \cS(\rho^{-1}(\alpha))$ be the closest point to $g_e(-t_e)$ (and hence closest to the entire geodesic $g_e$).   Then  ${f}_e(Y_e) =D_\alpha^e \rho(Y_e) = \rho(Y_e)$, and hence
\[ d_{\WP}({f}_e(Y_e),Y_e) = d_{\WP}(\rho(Y_e),Y_e) \geq d_{\WP}(X_0,\rho(X_0)).\]
Moreover, $\rho^{-1}(\alpha)=f_e^{-1}(\alpha)$ so the minimal length of $\rho^{-1}(\alpha)$ along $g_e$ is realized at time $-t_e$ and $\ell_{\rho^{-1}(\alpha)}(g_e(-t_e))\to 0$ as $e\to\infty$.  By Corollary~\ref{cor : dist to stratum} the distance between $g_e(-t_e)$ and $\cS(\rho^{-1}(\alpha))$ is bounded above by $\sqrt{2\pi \ell_{\rho^{-1}(\alpha)}(g_e(-t_e))}$, so we obtain
 \[d_{\WP}(Y_e,g_e(-t_e)) = d_{\WP}({f}_e(Y_e),g_e(0)) \to 0\]
  as $e \to \infty$. It follows then from the triangle inequality that
\begin{eqnarray*}
\liminf_{e \to \infty} t_e&=& \liminf_{e \to \infty} d_{\WP}(g_e(-t_e),g_e(0))\\
 &\geq& \liminf_{e\to\infty} \Big(d_{\WP}(Y_e,f_e(Y_e))-d_{\WP}(g_e(-t_e),Y_e)\\
 &-&d_{\WP}(f_e(Y_e),g_e(0))\Big)\\
 &\geq& d_{\WP}(X_0,\rho(X_0)).
 \end{eqnarray*}

On the other hand, since $g_e$ is the geodesic axis of $f_e$, $t_e$ is less than the distance that ${f}_e$ translates along $g_e^\omega$, which is precisely $|\omega|=d_{\WP}(X_0,\rho(X_0))$.  That is, $t_e \leq d_{\WP}(X_0,\rho(X_0))$, and hence
\[ \limsup_{e \to \infty} t_e \leq d_{\WP}(X_0,\rho(X_0)). \]
Combining this with the above, we have $\displaystyle{\lim_{e \to \infty} t_e = d_{\WP}(X_0,\rho(X_0))}$, completing the proof of the lemma.
\end{proof}

We are now ready for the proof of Proposition~\ref{prop : converge to omega}.
\begin{proof}[Proof of Proposition~\ref{prop : converge to omega}.]
We recall that $g_e$ intersects a fixed thick part of \OT\ space, independent of $e$, at the times $(k-1)t_e + t_e'$, for all $k \in \mathbb Z$. Denote the closest point on $g_e^\omega$ to the point $g_e((k-1)t_e + t_e')$ by $X_{e,k}$. The distance between $X_{e,k}$ and $g_e((k-1)t_e + t_e')$ must tend to zero as $e \to \infty$.  Otherwise, the strict negative curvature in the thick part of $\Teich(S)$ would imply a definite contraction factor $\delta < 1$ for the closest point projection to $g_e$ restricted to $g_e^\omega$ for all $e$ sufficiently large.  Since $X_{e,k+1}=f_e(X_{e,k})$, $d_{\WP}(X_{e,k},X_{e,k+1})=|\omega|$. Now by the contraction of the projection on $g_e$ and Lemma \ref{lem : lim te} we would have that
\[ |\omega| = \lim_{e \to \infty} t_e \leq \delta |\omega| \]
an obvious contradiction.  The sequence of points $\{X_{e,k}\}_{k \in \mathbb Z}$ is ${f}_e$--invariant and its distance to $g_e$ tends to $0$ as $e \to \infty$.  Appealing to the CAT(0) property of $\overline{\Teich(S)}$, the furthest point of $g_e^\omega$ to $g_e$ must also have distance tending to $0$, and hence the Hausdorff distance between $g_e$ and $g_e^\omega$ tends to $0$, as desired.
\end{proof}

\begin{cor}\label{cor : min in S(alpha)}
The point $X_0 \in \overline{\cS(\rho^{-1}(\alpha))}$ where the minimum of the function $F(X) = d_{\WP}(X,\rho(X))$ (restricted to $\overline{\cS(\rho^{-1}(\alpha))}$) is realized lies in $\cS(\rho^{-1}(\alpha))$.  Moreover,
\[ \lim_{e \to \infty} g_e([-t_e,0]) = \omega.\]
\end{cor}
\begin{proof}
First recall that $\ell_{f_e^{-1}(\alpha)}(g_e(-t_e))=\ell_{\rho^{-1}(\alpha)}(g_e(-t_e))$ and that $\ell_{\alpha}(g_e(0))$ goes to $0$ as $e\to\infty$. The distance between the point $g_e(-t_e)$ and the stratum $\cS(\rho^{-1}(\alpha))$ by Corollary~\ref{cor : dist to stratum} is bounded above by $\sqrt{2\pi \ell_{\rho^{-1}(\alpha)}(g_e(-t_e))}$, and hence tends to zero. Thus the point $g_e(-t_e)$ converges to the closure of $\cS(\rho^{-1}(\alpha))$. From Theorem~\ref{thm : short curves along ge}, the only curve which is very short (has length less than $\ep$) at $g_e(-t_e)$ is $\rho^{-1}(\alpha)$, so the point $g_e(-t_e)$ converges to $\cS(\rho^{-1}(\alpha))$. Similarly we can see that $g_e(0)$ converges to $\cS(\alpha)$.

Moreover, since $g_e$ is a geodesic and $g_e(0)=\rho(g_e(-t_e))$ the point $g_e(-t_e)$ converges to $X_0$ at which the minimum of the function $F$ is realized. Also, $g_e(0)$ converges to $\rho(X_0)$.
 
 By the non-refraction of property of WP geodesics (Theorem~\ref{thm : non-ref}) then the interior of $\omega$ lies in $\Teich(S)$. The limiting behavior of the geodesic follows from the CAT(0) property of the metric on $\overline{\Teich(S)}$.
\end{proof}

The geodesic axis $g_e$ descends to a closed geodesic $\hat g_e$ in ${\mathcal M}(S_{0,p})$ and $\omega$ descends to a geodesic segment $\hat \omega$ in $\overline{\mathcal M(S_{0,p})}$.  The previous corollary immediately implies the following.
\begin{cor} \label{cor : limit ge in moduli}
As $e \to \infty$, we have convergence $\hat g_e \to \hat \omega \subset \overline{\mathcal M(S_{0,p})}$.
\end{cor}

For any $e> 0$, we let $\delta_e$ denote the geodesic segment from the midpoint of $\omega$ to the midpoint of ${f}_e(\omega)$.  We also let $\omega^-$ and $\omega^+$ denote the first and second half-segments of $\omega$, respectively (so $\omega = \omega^- \cup \omega^+$ and $\omega^- \cap \omega^+$ is the midpoint of $\omega$).  Our construction in the next section will use the following.
\begin{lem}  \label{lem : small angle for consecutive segments} 
Given $\epsilon > 0$, there exists $N > 0$ so that for all $e \geq N$, the triangle with sides $\delta_e$,  $\omega^+$, and $f_e(\omega^-)$ has angles less than $\epsilon$ at the endpoints of $\delta_e$, and the Hausdorff distance between $\delta_e$ and $\omega^+ \cup f_e(\omega^-)$ is at most $\epsilon$.
\end{lem}
\begin{proof}  By Proposition~\ref{prop : converge to omega} and Corollary~\ref{cor : min in S(alpha)}, the segment $g_e([-\frac{t_e}{2},\frac{t_e}{2}])$ can be made as close as we like to $\omega^+ \cup {f}_e(\omega^-)$.  Since $g_e([-\frac{t_e}{2},\frac{t_e}{2}])$ and $\delta_e$ are both geodesics in a CAT(0) space, and since their endpoints become closer and closer as $e$ tends to infinity, it follows that the distance between $g_e([-\frac{t_e}{2},\frac{t_e}{2}])$ and $\delta_e$ tends to zero as $e \to \infty$.   Therefore, the distance between $\delta_e$ and $\omega^+ \cup {f}_e(\omega^-)$ tends to zero as $e \to \infty$.  This proves the second statement of the lemma.  

Short initial segments of $\delta_e$ and $\omega^+$ are both geodesics in a Riemannian manifold, they have a common initial point, and the initial segment of $\delta_e$ converges to that of $\omega^+$ as $e \to \infty$.  It follows that the angle between $\delta_e$ and $\omega^+$ tends to zero as $e \to \infty$.  A similar argument (composing with ${f}_e^{-1}$) shows that the angle at endpoint of $\delta_e$ and ${f}_e(\omega^-)$ tends to zero as $e \to \infty$.  This proves the first statement of the lemma.
\end{proof}

%%%%%%%%%%%%%%%%%%%%%%%%
%%%%%%%%%%%%%%%%%%%%%
\section{The non-uniquely ergodic case}\label{sec : NUE case}
%%%%%%%%%%%%%%%%%%%%%
%%%%%%%%%%%%%%%%%%%%%%%%

%\bm{I added this paragraph to explain what is going on in this section}
Given a sequence of integers $\mathcal{E}\subset \mathbb{N}$ in this section first we construct a WP geodesic ray $r$ that is strongly asymptotic to the piecewise geodesic $g^{\omega}_{\mathcal{E}}$ in $\overline{\Teich(S)}$ similar to the construction in \S\ref{sec : closed geodesics}, but now for a non-constant sequence $\mathcal E$; see (\ref{eq:g-omega-E}). The proof of strong asymptoticity involves producing regions with definite total negative curvature on ruled surfaces and an application of the Gauss-Bonnet Theorem (c.f.~\cite{bmm2,asympdiv}).
The asymptoticity to $g^\omega_{\mathcal{E}}$ helps us to develop good control on lengths of curves along $r$ in \S\ref{subsec : curves on r}  and determine the limit set of $r$ in the Thurston compactification of Teichm\"uller space in \S\ref{subsec : limit sets}. 
In \S\ref{subsec : main reduction} we prove a technical result required for determining the limit sets of rays in \S\ref{subsec : limit sets}.

\subsection{Infinite geodesic ray}
Consider a sequence $\mathcal E = \{e_k\}_k\subset \mathbb{N}$ with $e_0 > E$ and $e_{k+1} \geq a e_k$ for some $a > 1$ and all $k$, to which we will impose further constraints later.  We write $\phi_k$, $\Phi_k = \phi_1 \cdots \phi_{k-1} \phi_k$, and $\Gamma(\mathcal E) = \{\gamma_k\}_{k=0}^\infty$ as in (\ref{eq : phi}) and (\ref{eq : gamma}) in \S\ref{sec : sequence of curves on S0p}.   Recall from (\ref{eqn : alt def gamma}) that $\gamma_k = \Phi_{k-m}(\alpha)$ for all $k \geq m$. Moreover, recall that the sequence $\{\gamma_k\}_k$ converges to a minimal non-uniquely ergodic lamination $\nu$ in $\mathcal{EL}(S)$ by Corollary \ref{cor : sequence to ending lamination}.

Let $\omega$ denote the Weil-Petersson geodesic segment connecting the point $X_0 \in \cS(\rho^{-1}(\alpha))$ to $\rho(X_0) \in \cS(\alpha)$ as in \S\ref{sec : closed geodesics}.  Note that $X_0 \in \cS(\rho^{-1}(\alpha))$ and so
\begin{equation} \label{eq : where PhikX0 is} \Phi_k(X_0) \in \cS(\Phi_k(\rho^{-1}(\alpha))) = \cS(\Phi_{k-1}(\alpha)) = \cS(\gamma_{k+m-1}).
\end{equation}
Write $\delta_k$ for the geodesic segment connecting midpoints of $\omega$ and $\phi_k(\omega)$ (compare with \S\ref{sec : closed geodesics} and Lemma~\ref{lem : small angle for consecutive segments}).  The endpoint of $\delta_k$ on $\omega$ will be called its initial endpoint, and the one on $\phi_k(\omega)$ its terminal endpoint.  The image of $\delta_k$ under any mapping class will have its endpoints labelled as initial and terminal according to those of $\delta_k$.

With this notation, we claim that the terminal endpoint of $\Phi_k(\delta_{k+1})$ is the same as the initial endpoint of $\Phi_{k+1}(\delta_{k+2})$.  Indeed, applying $\Phi_k^{-1}$ to this pair of arcs, we have $\delta_{k+1}$ and $\phi_{k+1}(\delta_{k+2})$.  The terminal endpoint of $\delta_{k+1}$ is the midpoint of $\phi_{k+1}(\omega)$.  This is the $\phi_{k+1}$--image of the midpoint of $\omega$, which is also the $\phi_{k+1}$--image of the initial endpoint of $\delta_{k+1}$, as claimed.

Concatenating segments of this type defines a half-infinite path:
\begin{equation}\label{eq:RE} R_{\mathcal E} = \delta_1 \cup \Phi_1(\delta_2) \cup \Phi_2(\delta_3) \cup \Phi_3(\delta_4) \cup \cdots.\end{equation}
This path fellow-travels the concatenation of $\omega$ and its translates:
\begin{equation}\label{eq:g-omega-E} g_{\mathcal E}^\omega = \omega \cup \Phi_1(\omega) \cup \Phi_2(\omega) \cup \Phi_3(\omega) \cup \cdots. \end{equation}
By (\ref{eq : where PhikX0 is}), projecting $R_{\mathcal E}$ and $g_{\mathcal E}^\omega$ to the curve complex (via the systole map) gives paths fellow traveling $\{\gamma_k\}_{k=0}^\infty$.  By Proposition~\ref{prop : sequence on S_0,p} and Theorem~\ref{thm : subsurface coeff estimate}, it follows that these are quasi-geodesics in the curve complex.  Since the projection to the curve complex is coarsely Lipschitz, so $R_{\mathcal E}$ and $g_{\mathcal E}^\omega$ are also quasi-geodesics.

We will also be interested in a truncation of $R_{\mathcal E}$ after $k$ steps:
\[ R_{\mathcal E}^k = \delta_1 \cup \Phi_1(\delta_2)\cup \cdots \cup \Phi_{k-1}(\delta_k),\]
and let $r_k$ denote the geodesic segment connecting the initial and terminal point of the broken geodesic segment $R_{\mathcal E}^k$; see Figure~\ref{fig : GB}.

The angle between consecutive segments $\Phi_{k-1}(\delta_k)$ and $\Phi_k(\delta_{k+1})$ will be denoted $\theta_k$.  Applying $\Phi_{k-1}$, this is the same as the angle between $\delta_k$ and $\phi_k(\delta_{k+1})$.  Observe that the angle $\theta_k$ is at least $\pi$ minus the sum of the angle  between $\delta_k$ and $\phi_k(\omega)$ and $\phi_k(\delta_{k+1})$ and $\phi_k(\omega)$ (with appropriate directions chosen).  Since $\phi_k = D_\alpha^{e_{k+m-1}}\rho = f_{e_{k+m-1}}$, by taking $e_{k+m-1}$ and $e_{k+m}$ sufficiently large, appealing to Lemma~\ref{lem : small angle for consecutive segments} we can ensure that $\theta_k$ is as close to $\pi$ as we like.  In particular, we additionally assume that our sequence $\{e_k\}_k$ grows fast enough that
\begin{equation}\label{eq : fast growth} 
\sum_{k=1}^\infty \pi - \theta_k < 1.
\end{equation}
We can also (clearly) assume that the integers $e_k$ are all large enough so that $\theta_k \geq \frac{\pi}{2}$ for all $k$.

\begin{remark} While we have imposed growth conditions here to control angles, it is worth mentioning that these are in addition to those conditions already imposed to prove non-unique ergodicity.
\end{remark}

%Note that in the above paragraph we implicitly imposed a growth condition on the sequence $\mathcal{E}$.

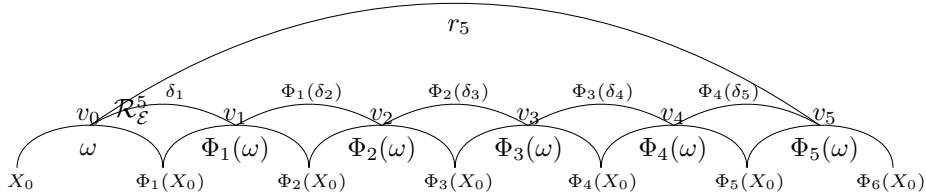
\begin{figure} [htb]
\labellist
\small\hair 2pt
\pinlabel $^{X_0}$ [b] at 7 -8
\pinlabel $^{\Phi_1(X_0)}$ [b] at 65 -8
\pinlabel $^{\Phi_2(X_0)}$ [b] at 120 -8
\pinlabel $^{\Phi_3(X_0)}$ [b] at 177 -8
\pinlabel $^{\Phi_4(X_0)}$ [b] at 233 -8
\pinlabel $^{\Phi_5(X_0)}$ [b] at 290 -8
\pinlabel $^{\Phi_6(X_0)}$ [b] at 345 -8
\pinlabel $^{\delta_1}$ [b] at 67 27
\pinlabel $^{\Phi_1(\delta_2)}$ [b] at 120 27
\pinlabel $^{\Phi_2(\delta_3)}$ [b] at 177 27
\pinlabel $^{\Phi_3(\delta_4)}$ [b] at 233 27
\pinlabel $^{\Phi_4(\delta_5)}$ [b] at 282 27
\pinlabel $r_5$ [b] at 177 55
\pinlabel $\mathcal{R}^5_{\mathcal{E}}$ [b] at 50 20
\pinlabel $\omega$ [b] at 33 8
\pinlabel $\Phi_1(\omega)$ [b] at 90 5
\pinlabel $\Phi_2(\omega)$ [b] at 147 5
\pinlabel $\Phi_3(\omega)$ [b] at 204 5
\pinlabel $\Phi_4(\omega)$ [b] at 260 5
\pinlabel $\Phi_5(\omega)$ [b] at 319 5
\pinlabel $v_0$ [b] at 33 20
\pinlabel $v_1$ [b] at 90 20
\pinlabel $v_2$ [b] at 147 20
\pinlabel $v_3$ [b] at 204 20
\pinlabel $v_4$ [b] at 260 20
\pinlabel $v_5$ [b] at 319 20
\endlabellist
\begin{center}
\includegraphics[width=12cm]{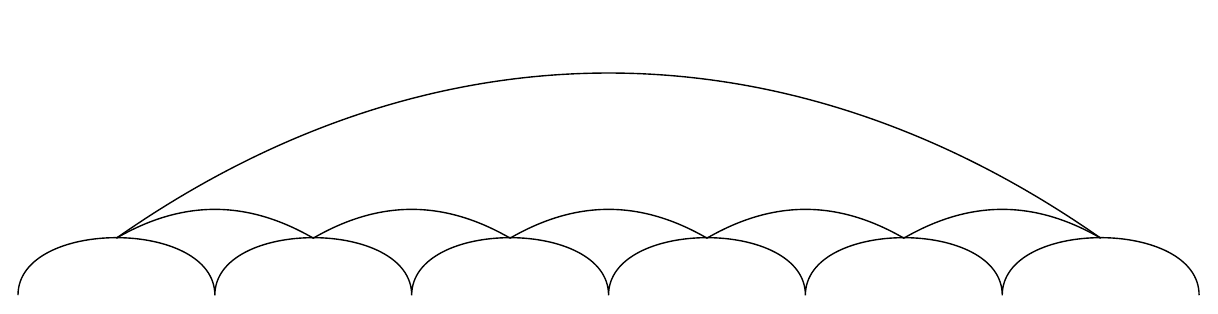}

\caption{The concatenation of geodesic segments $\delta_1,\Phi_1(\delta_2),\ldots,\Phi_4(\delta_5)$ defining $R_{\mathcal E}^5 \subset R_{\mathcal E}$ 
and $\omega\cup\Phi_1(\omega)\cup\ldots\cup\Phi_5(\omega)\subset g^\omega_{\mathcal E}$, together with the geodesic segment $r_5$ connecting the endpoints of $R_{\mathcal E}^5$.}
\label{fig : GB}
\end{center}
\end{figure}

\begin{prop} \label{prop : limit geodesic}
The geodesic segments $r_k$ limit to a geodesic ray $r$ as $k\to\infty$, and all three of $r$, $R_{\mathcal E}$, and $g^\omega_{\mathcal E}$ are strongly asymptotic (the distance between any pair of them tends to zero).
\end{prop}
\begin{proof}  According to the last part of Lemma~\ref{lem : small angle for consecutive segments}, $R_{\mathcal E}$ and $g^\omega_{\mathcal E}$ are strongly asymptotic.  Therefore, it suffices to prove that $r_k$ has a limit $r$, and that this is asymptotic to $R_{\mathcal E}$.  Before proceeding, we note that $\Phi_k(X_0)$ lies in the stratum $\cS(\Phi_k(\rho^{-1}(\alpha))) = \cS(\Phi_{k-1}(\alpha)) = \cS(\gamma_{k+m})$

Let $\{v_i\}_{i=0}^\infty$ denote the concatenation points of $R_{\mathcal E}$.  Denote by $P_k$ a ruled polygon bounded by $r_k$ and $R_{\mathcal E}^k$.  This polygon has vertices $v_0,\ldots,v_k$.   Let $\theta_i^k$ denote the interior angles of $P_k$ at $v_i$, for $i = 0,\ldots,k$, and observe that for $0 < i < k$, we have $\theta_i \leq \theta_i^k$.  In addition, there are constants $c_0 < 0$ and $d_0 > 0$ so that the $d_0$--neighborhood of $v_i$ in $P_k$ has Gaussian curvature $K \leq c_0$.  Consequently, for any $d < d_0$, if $r_k$ is disjoint from the $d$--neighborhood $N_d(v_i)$ of $v_i$, then since $\theta_i^k \geq \frac{\pi}2$, $N_d(v_i) \cap P_k$ contains a quarter-sector of a disk of radius $d$ centered at $v_i$ in a surface of curvature at most $c_0$.  Therefore, the integral of the curvature $K$ over $N_d(v_i) \cap P_k$ satisfies
\[ \int_{N_d(v_i) \cap P_k} K \, dA \leq \frac{c_0 \pi d^2}{4}.\]
By the Gauss-Bonnet Theorem (see e.g. \cite[Theorem V.2.5]{chaveldg}), we have
\[ \int_{P_k} K dA + \sum_{i=0}^k (\pi - \theta_i^k) = 2 \pi \]
which implies
\[ \theta_0^k + \theta_k^k - \sum_{i=1}^{k-1} (\pi - \theta_i^k) =  \int_{P_k} K dA .\]
For any $d > 0$, let $i_1,i_2,\ldots,i_j$ denote those indices $i$ for which $R_{\mathcal E}^k$ is more than $d$ away from $v_i$.  Then by our assumption on the angles $\theta_i$ in (\ref{eq : fast growth}) we have

\begin{eqnarray*} 
\theta_0^k +\theta_k^k - 1 \leq  \theta_0^k + \theta_k^k - \sum_{i=1}^{k-1} (\pi - \theta_i^k) &=&  \int_{P_k} K dA \\
&\leq& \sum_{\ell = 1}^j \int_{N_{v_\ell} \cap P_k} K dA  \leq \frac{ j c_0 \pi d^2}{4}.
\end{eqnarray*}
Since $c_0 < 0$, this implies
\[ j \leq \frac{4(\theta_0^k + \theta_k^k - 1)}{c_0 \pi d^2} \leq \frac{4|1 - (\theta_0^k + \theta_k^k)|}{|c_0| \pi d^2}\leq \frac{4(1 + 2\pi)}{|c_0| \pi d^2} .\]

This bounds the number of vertices along $R_{\mathcal E}^k$ that can be further than $d$ away from $r_k$ by some number $J(d)$, which is independent of $k$.  Therefore, for any $N > 0$ and $k,k' \geq N+2J(d) + 1$, there is a vertex $v_i$ of $R_{\mathcal E}$ with $N \leq i \leq \min\{k,k'\}$ so that $r_k$ and $r_{k'}$ contain points $x_k$ and $x_{k'}$, respectively, which are within distance $d$ of $v_i$.  Therefore, $x_k$ and $x_{k'}$ are within distance $2d$ of each other.  Since $R_{\mathcal E}$ is a quasi-geodesic, the distance from $v_i$ to $v_0$ tends to infinity with $i$.  Consequently, as $N$ tends to infinity, the distance from $x_k$ and $x_{k'}$ to $v_0$ also tends to infinity.  By convexity of the distance function  between two geodesic segments in a $\CAT(0)$ space, it follows that for any $D > 0$, the initial segments of $\{r_k\}_k$ of length $D$ form a Cauchy sequence in the topology of uniform convergence.  By completeness of $\overline{\Teich(S)}$, these segments of length $D$ converge. Letting $D$ tend to infinity, it follows that $r_k$ converges (locally uniformly) to a geodesic ray $r$.

For any $d > 0$, suppose that $v_i$ is a vertex of $R_{\mathcal E}$ further than $2d$ away from any point of $r$.  For $k$ sufficiently large, it follows that $r_k$ is further than $d$ from $v_i$. Since there are at most $J(d)$ of the latter indices $i$, it follows that $r$ must come closer than $2d$ from all but $J(d)$ vertices. In particular, there exists $N(d)$ so that for all $i \geq N(d)$, $r$ comes within $2d$ of $v_i$.  
By convexity of the distance between geodesics in the WP metric, the distance of any point on $\mathcal{R}_{\mathcal E}$ lying between consecutive vertices $v_i$ and $v_{i+1}$ (for $i \geq N(d)$) and $r$ is no more than $2d$.  Therefore, the tail of $R_{\mathcal E}$ starting at $v_{N(d)}$ is within Hausdorff distance $2d$ from some tail of $r$. Since $d$ was arbitrary, it follows that $R_{\mathcal E}$ and $r$ are strongly asymptotic, as required.
\end{proof} 

\iffalse

I removed this sentence because it is wrong
{\color{red} The angles at the vertices $v_i$ are at least $\frac{\pi}{2}$, so $R_{\mathcal E}$ is a convex subset of $\Teich(S)$. Now since the distance between a geodesic and a convex subset of a $\CAT(0)$ space is a convex function},
\fi

In the rest of this section let $r:[0,\infty)\to \Teich(S)$ be the geodesic ray from Proposition \ref{prop : limit geodesic}.

%%%%%%%%%%%%%%%%%%%%
\subsection{Curves along $r$}\label{subsec : curves on r}
%%%%%%%%%%%%%%%%%%

 The following lemma is a straightforward consequence of the setup of curves $\{\gamma_k\}_k$ in $\S 3$ and the choice of the segment $\omega$ in the previous section which we record  as a convenient reference.
\begin{lem} \label{lem : index record} For any $k \geq m-1$ the initial and terminal endpoints of $\Phi_{k-m+1}(\omega)$ are in the strata $\cS(\gamma_k)$ and $\cS(\gamma_{k+1})$, respectively.  Furthermore, for any compact subsegment $I \subset int(\omega)$, the $2m$ consecutive curves 
\[ \{\gamma_{k-m+1},\ldots,\gamma_{k+m}\} \]
have bounded length on $\Phi_{k-m+1}(I)$, with the bound depending on the choice of interval $I$, but independent of $k$.
\end{lem}
\begin{proof}
Recall that $\alpha = \rho^m(\gamma_0) = \rho(\gamma_{m-1})$, and hence $X_0 \in \cS(\alpha) =\cS(\gamma_{m-1})$.  
Consequently $\Phi_{k-m+1}(X_0) \in \cS(\gamma_k)$, since $\Phi_{k-m+1}(\gamma_{m-1}) = \gamma_k$; see (\ref{eqn : alt gamma def}).
Thus the initial endpoint of $\Phi_{k-m+1}(\omega)$ is in $\cS(\gamma_k)$.  Since the terminal endpoint of $\Phi_{k-m+1}(\omega)$ is the initial endpoint of $\Phi_{k-m+2}(\omega)$, this common endpoint lies in $\cS(\gamma_{k+1})$, proving the first statement.

The compact subsegment $I \subset int(\omega)$ is entirely contained in Teichm\"uller space, and hence the curves $\gamma_0,\ldots,\gamma_{2m-1}$ have bounded length in $I$.  Since the $\Phi_{k-m+1}$--image of these curves are precisely those listed in the lemma, the second statement also follows.
\end{proof}

\begin{thm}\label{thm : short times along r}
There exists a sequence $\{t_k\}_{k=1}^\infty$ which is eventually increasing, such that $\displaystyle{\lim_{k \to \infty} \ell_{\gamma_k}(r(t_k)) = 0}$.  Furthermore, for any $\epsilon > 0$ sufficiently small, the set of curves with length less than $\epsilon$ along $r$ is contained in $\{\gamma_k\}_{k=0}^\infty$ and contains a tail of this sequence, $\{\gamma_k\}_{k \geq N}$, for some $N = N(\epsilon) \in \mathbb Z$.
\end{thm}
\begin{proof}
Since $r$ is strongly asymptotic to $g^{\omega}_{\mathcal{E}}$ by Proposition \ref{prop : limit geodesic} we may choose $t_k$ so that
\[d_{\WP}(r(t_k),\cS(\gamma_k))\to0\]
as $k\to\infty$. Then by the formula 
$$d_{\WP}(r(t_k),\cS(\gamma_k))=\sqrt{2\pi \ell_{\gamma_k}(r(t_k))} +O\Big(\ell_{\gamma_k}(r(t_k))^{5/2}\Big)$$ 
from Corollary~\ref{cor : dist to stratum}, where the constant of the $O$ notation depends only on an upper bound for the length of $\gamma_k$ at the point $r(t_k)$, we see that
\[ \lim_{k \to \infty} \ell_{\gamma_k}(r(t_k)) = 0.\]
Since $g_{\mathcal E}$ passes through the strata $\{\cS(\gamma_k)\}_k$ in order (i.e.~$g^\omega_{\mathcal E}$ intersects $\cS(\gamma_k)$ before $\cS(\gamma_{k+1})$), the times when $r$ comes close to $\{ \cS(\gamma_k)\}_k$ also occur in order.  This proves the first statement.

For the second statement, we note that the first statement implies that for any $\epsilon > 0$, there exists $N(\epsilon) > 0$ so that for all $k \geq N(\epsilon)$, $\gamma_k$ has length less than $\epsilon$ at some point along $r$ (in fact, at the point $r(t_k)$).  Moreover, by Lemma~\ref{lem : index record} $\Phi_{k-m+1}(\omega)$ goes from $\cS(\gamma_k)$ to $\cS(\gamma_{k+1})$, and no other curves become very short along $\Phi_{k-m+1}(\omega)$.  Again appealing to the fact that $r$ is asymptotic to $g_{\mathcal E}^\omega$, it follows that for $k$ sufficiently large, the only curves of length less than $\epsilon$ on $r([t_k,t_{k+1}])$ are $\gamma_k$ and $\gamma_{k+1}$.  Therefore, for $\epsilon$ sufficiently small, the only curves that can have length less than $\epsilon$ along $r$ are from $\{\gamma_k\}_k$.
\end{proof}

Because $r$ is asymptotic to $g^\omega_{\mathcal E}$, there is a version of Lemma~\ref{lem : index record} for $r$.
We will consider sequences $\{s_k\}_k \subset [0,\infty)$ satisfying one of the following:

\begin{enumerate}
\item[(C1)] There exists $\epsilon > 0$ such that $t_k + \epsilon < s_k < t_{k+1}-\epsilon$, or
\item[(C2)] $\displaystyle{\lim_{k \to \infty} |t_{k+1} - s_k| = 0 }$.
\end{enumerate}

\begin{cor} \label{cor : lengths along r} Suppose that $\{s_k\}_k \subset [0,\infty)$ is a sequence.
\begin{itemize}
\item If $\{s_k\}_k$ satisfies (C1), then the $2m$ consecutive curves $\gamma_{k-m+1},\ldots,\gamma_{k+m}$ have bounded length at $r(s_k)$, independent of $k$, but depending on $\epsilon$.
\item If $\{s_k\}_k$ satisfies (C2), then $\displaystyle{\lim_{k\to\infty}\ell_{\gamma_{k+1}}(r(s_k)) = 0}$, and the $2m-1$ consecutive curves $\gamma_{k-m+2},\ldots,\gamma_{k+m}$ have bounded length at $r(s_k)$, independent of $k$.
\end{itemize}
\end{cor}
\begin{proof} Suppose that we are in case (C1).  Then there exists a compact interval $I \subset int(\omega)$ so that the Hausdorff distance between $\Phi_{k-m+1}(I)$ and $r([t_k+\epsilon,t_{k+1}-\epsilon])$ tends to zero as $k \to \infty$.  By Lemma~\ref{lem : index record}, $\gamma_{k-m+1},\ldots,\gamma_{k+m}$ have bounded length along $\Phi_{k-m+1}(I)$.  Since $\Phi_{k-m+1}(I)$ remains bounded away from the completion strata of $\Teich(S)$,  $\gamma_{k-m+1},\ldots,\gamma_{k+m}$ also have bounded length along $r([t_k+\epsilon,t_{k+1}-\epsilon])$, as required.

For case (C2), the assumptions imply that $d_{WP}(r(s_k),\cS(\gamma_{k+1})) \to 0$ as $k \to \infty$, and hence $\ell_{\gamma_{k+1}}(r(s_k)) \to 0$ as $k\to\infty$.  The bound on the lengths of $\gamma_{k-m+2},\ldots,\gamma_{k+m}$ follows from case (C1) and convexity of the length-functions (\cite{wolb}).  Indeed, from case (C1), we know that the curves $\gamma_{k-m+2},\ldots,\gamma_{k+m}$ have uniformly bounded lengths at $r(\tfrac{t_k+t_{k+1}}2)$ and $r(\tfrac{t_{k+1}+t_{k+2}}2)$, and hence all the curves have uniformly bounded length along $r([\tfrac{t_k+t_{k+1}}2,\tfrac{t_{k+1}+t_{k+2}}2])$ by convexity of length-functions.
\end{proof}

As another application of Theorem~\ref{thm : short times along r}, we can identify the ending lamination of $r$.

\begin{cor} \label{cor : nu is end lam of r} The lamination $\nu$ is the ending lamination of the ray $r$.
\end{cor}
\begin{proof} By Theorem~\ref{thm : short times along r}, $\displaystyle{\lim_{k \to \infty} \ell_{\gamma_k}(r(t_k)) = 0}$.  Since, by Theorem \ref{thm : nue}, the subsequence
\[ \{\gamma_k \mid k \equiv 0 \mbox{ mod } m\}\]
converges to $\bar\nu_0$ in $\ML(S)$ (after appropriately scaling), it follows that the ending lamination of $r$ contains $\nu$. 
Moreover, $\nu\in\EL(S)$, and hence $\nu$ is the ending lamination of $r$.
\end{proof}

\iffalse
\bm{ the new lemma is inserted here}
\begin{lem}\label{lem : nbt}
For any curve $\gamma_k$ and any proper subsurface $Z$ so that $\gamma_k\pitchfork Z$, we have
\[d_Z(\mu,\gamma_k)+d_Z(\gamma_k,\nu)\asya d_{Z}(\mu,\nu)\]
\end{lem}
\begin{proof}
By the triangle inequality in $\mathcal{C}(Z)$ we have
\[d_Z(\mu,\nu) \leq d_Z(\mu,\gamma_k)+d_Z(\gamma_k,\nu).\] 
So we only need to prove the inequality in the other direction, which is
\begin{equation}\label{eq : nbt} d_Z(\mu,\nu) \stackrel +\succ d_Z(\mu,\gamma_k)+d_Z(\gamma_k,\nu).\end{equation}
First, suppose that $Z\neq Y_{\gamma_i}$ for all $i\geq 1$. Then according to Theorem \ref{thm : subsurface coeff estimate}, the two subsurface coefficients on the right-hand side are bounded by $R(\mu)$ and $R$, respectively. So the quasi-inequality holds for additive error $R+R(\mu)$.

Now suppose that $Z=Y_{\gamma_i}$ for some $i\geq 1$. Then either $i<k$ and $\gamma_i\pitchfork\gamma_k$ or $i>k$ and $\gamma_i\pitchfork\gamma_k$. In the former case, by Theorem \ref{thm : subsurface coeff estimate}, the subsurface coefficient on the left-hand side of (\ref{eq : nbt}) is at least $e_i-R(\mu)$. Moreover the two subsurface coefficients on the right-hand side are bounded by $e_i+R(\mu)$ and $R$, respectively. Thus the quasi-inequality (\ref{eq : nbt}) holds for additive error $R+2R(\mu)$. Similarly we can see that the quasi-inequality holds for additive error $R+2R(\mu)$ in the later case. 

Therefore, the comparison of the theorem holds for the additive error $R+2R(\mu)$.
 \end{proof}
\fi

%%%%%%%%%%%%%%%%%%%
\subsection{The Limit set}\label{subsec : limit sets}
%%%%%%%%%%%%%%%%%%

By Corollary \ref{cor : nu is end lam of r} the ending lamination of $r$ is the minimal non-uniquely ergodic lamination $\nu$.
Let $\bar{\nu}^h,\; h=0,\ldots, m-1,$ be the ergodic measures supported on $\nu$ as in Theorem \ref{thm : nue}. 
Theorem~\ref{thm : NUE main vague} follows immediately from the following theorem.

\begin{thm}\label{thm : limgE 1-sk}
The limit set of $r$ in $\mathcal{PML}(S)$ is the concatenation of the edges 
\[\big[[\bar{\nu}^{0}],[\bar{\nu}^{1}]\big],\ldots, \big[[\bar{\nu}^{m-1}],[\bar{\nu}^{0}]\big]\]
 in the $1-$skeleton of the simplex of projective measures supported on $\nu$.
\end{thm}

We will reduce this to a more technical statement, and then in the next subsection, prove that technical statement.  As we will be exclusively interested in lengths of curves along $r$, for any curve $\delta$ and $s \in [0,\infty)$, we write
\[ \ell_\delta(s)= \ell_\delta(r(s)).\]
Our main technical result is the following theorem.

\begin{thm} \label{thm : main reduction}  Suppose that $\{s_k\}_k \subset [0,\infty)$ is a sequence.
\begin{itemize}
\item If $\{s_k\}_k$ satisfies (C1), then there exists $x_k > 0$ such that for any simple closed curve $\delta$, we have 
\[ \lim_{k \to \infty} \frac{x_k \I (\delta,\gamma_k)}{\ell_\delta(s_k)} = 1.\]
\item If $\{s_k\}_k$ satisfies (C2), then there exist $x_k,y_k \geq 0$ with $x_k+y_k > 0$ such that for any simple closed curve $\delta$, we have
\[ \lim_{k \to \infty} \frac{x_k \I (\delta,\gamma_k) + y_k \I(\delta,\gamma_{k+1})}{\ell_\delta(s_k)} = 1.\]
\end{itemize}
\end{thm}
\begin{proof}[Proof of Theorem~\ref{thm : limgE 1-sk} assuming Theorem~\ref{thm : main reduction}.] We will pass to subsequences in the following, and to avoid double subscripts, for a subsequence of a sequence $\{c_k\}_{k=1}^\infty$, we simply write $\{c_k\}_{k \in J}$, where $J$ is the index set defining the subsequence.  Likewise $\displaystyle{\lim_{k \in J} c_k}$ will denote the limit of the subsequence as the indices from $J$ tend to infinity.

Now suppose that $[\bar \mu] \in \PML(S)$ is a limit point of the ray $r$.  That is, for some sequence of times $\{s_j\}_j \subset [0,\infty)$ and any two curves $\delta,\delta'$, we have
\[ \lim_{j \to \infty} \frac{\ell_\delta(s_j)}{\ell_{\delta'}(s_j)} = \frac{\I(\delta,\bar \mu)}{\I(\delta',\bar \mu)}\]
(see $\S$\ref{subsec : Thcpct}). Since $s_j$ must tend to infinity as $j \to \infty$, by passing to a subsequence, we may assume that the sequence is increasing, and there exists an increasing sequence $\{k_j\}_j$ such that either $|s_j-t_{k_j}| \to 0$ or else there exists $\epsilon > 0$ so that $t_{k_j}+\epsilon < s_j < t_{k_j+1}-\epsilon$.  Consequently, after reindexing, we assume (as we may) that our sequence is a subsequence $\{s_k\}_{k \in J}$ of some sequence $\{s_k\}_{k=1}^\infty$ satisfying either (C1) or (C2).

Suppose first that $\{s_k\}_{k=1}^\infty$ satisfies (C1), and pass to a further subsequence (with index set still denoted $J$ for simplicity) so that all $k \in J$ are congruent to some $h \in \{0,\ldots,m-1\}$ mod $m$.  Then by our assumption and Theorem~\ref{thm : main reduction} we have
\[ \frac{\I(\delta,\bar \mu)}{\I(\delta',\bar \mu)} = \lim_{k \in J} \frac{\ell_\delta(s_k)}{\ell_{\delta'}(s_k)}
 = \lim_{k \in J} \frac{x_k\I(\delta,\gamma_k)}{x_k\I(\delta',\gamma_k)} = \frac{\I(\delta,\bar \nu^h)}{\I(\delta',\bar \nu^h)}
\]
where the last equality follows from the fact that $[\gamma_k] \to [\bar \nu^h]$ in $\PML(S)$, for $k \in J$ by Theorem \ref{thm : nue}.  But this implies that $[\bar \mu] = [\bar \nu^h]$ since $\delta,\delta'$ were arbitrary.

We further observe that if $h \in \{0,\ldots,m-1\}$, then $\{s_k:= \frac{t_k+t_{k+1}}2\}_k$ satisfies (C1), and the computations just given show that for any subsequence $\{s_k\}_{k \in J}$ such that $k \equiv h$ mod $m$ for all $k \in J$, we have $\displaystyle{\lim_{k \in J} r(s_k) = [\bar \nu^h]}$ in the Thurston topology.  Consequently, all the vertices of the simplex are in fact accumulation points.

Next, suppose that $\{s_k\}_{k=1}^\infty$ satisfies (C2), and again pass to yet another subsequence so that all $k \in J$ are congruent to some $h \in \{0,\ldots,m-1\}$ mod $m$.  In this case, we must pass to yet another subsequence so that $[x_k \gamma_k + y_k \gamma_{k+1}]$ converges to some $[\bar \mu'] \in \PML(S)$, for $k \in J$.  Note that, since by Theorem \ref{thm : nue} $[\gamma_k] \to [\bar \nu^h]$ and $[\gamma_{k+1}] \to [\bar \nu^{h+1}]$ (where we replace $h+1$ by $0$, if $h+1 = m$), we have  $[\bar \mu'] \in [[\bar \nu^h],[\bar \nu^{h+1}]]$.  Then, by similar reasoning we have
\begin{eqnarray*} \frac{\I(\delta,\bar \mu)}{\I(\delta',\bar \mu)} & = & \lim_{k \in J} \frac{\ell_\delta(s_k)}{\ell_{\delta'}(s_k)} = \lim_{k \in J} \frac{x_k\I(\delta,\gamma_k) + y_k\I(\delta,\gamma_{k+1})}{x_k\I(\delta',\gamma_k) + y_k\I(\delta',\gamma_{k+1})}\\
& = & \lim_{k \in J} \frac{\I(\delta,x_k\gamma_k+ y_k\gamma_{k+1})}{\I(\delta',x_k\gamma_k + y_k \gamma_{k+1})} = \frac{\I(\delta,\bar \mu')}{\I(\delta',\bar \mu')}.
\end{eqnarray*}
Here the second to last equality follow from bilinearity of intersection number, while the last equality follows since $[x_k \gamma_k + y_k \gamma_{k+1}] \to [\bar \mu']$ in $\PML(S)$, for $k \in J$.  Thus again we see that $[\bar \mu] = [\bar \mu']$.

So, the limit set of $r$ is contained in the required loop in the $1$--skeleton of the simplex of projective classes of measures on $\nu$.  If we fix $h \in \{0,\ldots,m-1\}$ and consider the arcs
\[ \{ r([\tfrac{t_k + t_{k+1}}{2},\tfrac{t_{k+1} + t_{k+2}}2]) \mid k \equiv h \mbox{ mod } m \} \]
it follows that the initial endpoints converge to $[\bar \nu^h]$ while the terminal endpoints converge to $[\bar \nu^{h+1}]$ (again replacing $h+1$ with $0$ if $h+1=m$).  Moreover, the accumulation set of this sequence of arcs is a connected subset of $[[\bar \nu^h],[\bar \nu^{h+1}]]$.  Any such set is necessarily the entire $1$--simplex.  Therefore, the ray $r$ accumulates on the entire loop, as required.
\end{proof}

%%%%%%%%%%%%%%%%%%%%%%%%%%%%%%%%%%
\subsection{Proof of Theorem~\ref{thm : main reduction}}\label{subsec : main reduction}
%%%%%%%%%%%%%%%%%%%%%%%%%%%%%%%%%%

Here we prove the required technical theorem used in the proof of Theorem~\ref{thm : limgE 1-sk}.  Throughout what follows, we assume that $\{s_k\}_k$ satisfies (C1) or (C2).  Many of the estimates can be carried out for both cases simultaneously.

From Corollary~\ref{cor : lengths along r}, the $2m-1$ curves $\gamma_{k-m+2},\ldots,\gamma_{k+m}$ have bounded lengths in $r(s_k)$, and since
\[ \gamma_{k-m+2},\ldots,\gamma_k,\gamma_{k+2},\ldots,\gamma_{k+m} \]
fill $S - \gamma_{k+1}$, there is a pants decomposition $P_k$ containing the $m$--component multicurve
\[ \sigma_k := \gamma_k \cup \ldots \cup \gamma_{k+m-1} \]
such that $\ell_{\beta}(s_k)$ is bounded for all $\beta \in P_k$, independent of $k$ (though the bounds depend on $\epsilon$ in case (C1)).
Write $P_k^c = P_k \setminus \sigma_k$.
%\cl{should make $\alpha$ into a $\beta$ here and throughout the rest of this subsection to avoid confusion with our previous usage of $\alpha$.}

For an arbitrary curve $\delta$ and a curve $\beta \in P_k$, the {\em contribution to the length of $\delta$ from $\beta$ in $r(s_k)$} is defined by the equation:
\begin{equation}\label{eq : l_del(s,al)} \ell_{\delta}(s_k,\beta)  := \I(\delta,\beta)\Big(w_{s_k}(\beta) + \tw_{\beta}(\delta,s_k)\ell_{\beta}(s_k)\Big) \end{equation}
where $\tw_{\beta}(\delta,s_k)$ is the twist of $\delta$ about $\beta$ at $r(s_k)$ as is defined in (\ref{eq : twist}), and $w_{s_k}(\beta)$ is the width of the largest embedded tubular neighborhood of $\beta$ in $r(t_k)$ (i.e.~the minimal distance between boundary components of the neighborhood).  By \cite[\S 4.1]{buser}, we have
\begin{equation} \label{eqn : collar lemma} w_{s_k}(\beta) = 2\log\left(\tfrac{1}{\ell_{\beta}(s_k)}\right). \end{equation}

The following estimate for the hyperbolic length of a curve $\delta$ from \cite[Lemmas 7.2, 7.3]{lineminimateichgeod} will be our primary tool.
\begin{thm} \label{thm : CRS} Suppose that the sequence $\{s_k\}_k$ satisfies (C1) or (C2).  Then, for any curve $\delta$ we have
\begin{equation}\label{eq : length expansion}
\Big|\ell_\delta(s_k)-\sum_{\beta\in P_k} \ell_\delta(s_k,\beta) \Big|= O\Big(\sum_{\beta\in P_k}\I(\delta,\beta)\Big).
\end{equation}
Here the constant of the $O$ notation depends only on the upper bound for the length of the curves in $P_k$.
\end{thm}

The proof of Theorem~\ref{thm : main reduction} now follows from estimating various terms in the sum in the above theorem, and finding that one (in case (C1)) or two (in case (C2)) dominate not only the other terms, but also the error term on the right. 

Recall that for any simple closed curve $\delta$, Theorem~\ref{thm : intersection} implies that for all $j$ sufficiently large we have
\begin{equation} \label{eqn : intersection estimates end} \I(\delta,\gamma_j) \stackrel{*}{\asymp} A(0,j)
\end{equation}
where the multiplicative error depends on $\delta$, but not on $j$.  
Combining (\ref{eqn : intersection estimates end}) and Lemma~\ref{lem : intersection estimate ratios}, we see that for all $0 \leq h \leq m-1$, we have
\begin{equation} \label{eqn : easy pants intersection estimate} \lim_{k \to \infty} \frac{\I(\delta,\gamma_{k+h})}{A(0,k+m)} = 0.
\end{equation}
Observe that the curves $\gamma_{k+h}$ here are precisely the components of $\sigma_k$.  It turns out that the intersection numbers with the other curves in $P_k$ (not just those in $\sigma_k$), are also controlled by $A(0,k+m)$.  This is essentially the Weil-Petersson analogue of \cite[Theorem 9.15]{nue2}.

%\bm{I changed $\alpha_k$ to $\beta_k$ for a curve in $P_k$ in the statement of the lemma and afterwards in this section, the index $k$ is superfluous and makes reading hard.}
%\cl{I think subscript $k$ should be put back in (and curve should be called $\beta_k$).  Otherwise one forgets that the curve is changing with $k$...}
\begin{lem} \label{lem : asymptotic intersection off gamma-k} For any $\beta_k \in P_k$ we have
\[ \lim_{k \to \infty} \frac{\I(\delta,\beta_k)}{A(0,k+m)} = 0.\]
\end{lem}
\begin{proof} By (\ref{eqn : easy pants intersection estimate}) it suffices to prove the lemma when $\beta_k \in P_k^c$, for all $k$.

%\bm{I also changes $Y_k$ to $Y$ and $Z_k$ to $Z$ in the following}
%\cl{I think this should be changed back}

Let $\mu$ be any fixed marking on $S$ and let $Y_k$ be the component of the complement $S \setminus \sigma_k$ that contains $\beta_k$.
\begin{claim}  There exists $I > 0$, depending only on $\mu$ and $\delta$ so that $\I(\pi_Y(\delta),\beta_k) \leq I$.
\end{claim}
 Since $\delta$ and $\mu$ are a fixed curve and marking, we can assume that their projections to all subsurfaces are uniformly close. Let $Z_k \subseteq Y_k$ be any subsurface with $\beta_k \pitchfork Z_k$, and observe that since $Z_k$ is disjoint from the $m$ consecutive curves in $\sigma_k$, Theorem~\ref{thm : subsurface coeff estimate} implies that it cannot be an annulus with core curve in the sequence $\{\gamma_i\}_i$.
By Corollary \ref{cor : lengths along r}, at the point $r(s_k)$ the $2m-1$ curves $\gamma_{k-m+2},\ldots,\gamma_{k+m}$ have length bounded independent of $k$, and hence $\I(\beta_k,\gamma_l)$ is uniformly bounded for each $l = k-m+2,\ldots,k+m$.   Since these curves fill $S\backslash\gamma_{k+1}$ and $\gamma_{k+1}\in\sigma_k$, $\pi_{Z_k}(\gamma_l) \neq \emptyset$ for some $k-m+2 \leq l \leq k+m$, and hence $d_{Z_k}(\gamma_l,\beta_k)$ is uniformly bounded.  Thus, by the triangle inequality and (\ref{eqn : not gamma_i, small projection}) we have
\[   d_{Z_k}(\beta_k,\delta) \asya d_{Z_k}(\gamma_l,\mu) \leq R(\mu). \]
Since this holds for all subsurfaces $Z_k \subseteq Y_k$, \cite[Corollary D]{compteichlip} tells us that $\I(\pi_{Y_k}(\delta),\beta_k)$ is uniformly bounded, as required.
\medskip

Every arc of $\pi_{Y_k}(\delta)$ comes from a pair of intersection points with curves in $\sigma_k$.  Consequently, taking $\kappa(\delta)$ as the second paragraph of Theorem~\ref{thm : intersection} and noting that $A(0,j)$ is increasing in $j$ we have
\[\I(\delta,\beta_k) \leq I \sum_{d = k}^{k+m-1}\I(\delta,\gamma_d) \asym_{\kappa(\delta)} I \sum_{d = k}^{k+m-1}A(0,d) \leq m I A(0,k+m-1).\]
Thus, setting $K = mI \kappa(\delta)$ the proof of lemma is complete.
\end{proof}

Next, we estimate the various terms of $\ell_\delta(s_k,\beta)$ for $\beta \in P_k$. 
%\cl{I'm not suggesting we write $\beta_k$ here... $\beta$ is fine, because we are not taking a limit with $k$}
\begin{lem} \label{lem : sufficient twist estimates}  Suppose that $\{s_k\}_k$ is a sequence satisfying either (C1) or (C2).  Then for all $k$ sufficiently large and $\beta \in P_k$, we have
\[ \tw_{\beta}(\delta,s_k) \asya  \tw_{\beta}(\gamma_0,s_k) \asya \left\{ \begin{array}{ll} e_k & \mbox{ if } \beta = \gamma_k\\\ 0 & \mbox{ } \beta \neq \gamma_k \mbox{ or } \gamma_{k+1}. \end{array} \right. \]
If $\{s_k\}_k$ satisfies (C1), then for all $k$ sufficiently large, $\tw_{\gamma_{k+1}}(\delta,s_k)\asya 0$.
\end{lem}
\begin{proof}   By Theorem~\ref{thm : subsurface coeff estimate} and Proposition~\ref{prop : sequence on S_0,p}, $\{\gamma_k\}_k$ is a quasi-geodesic ray in $\cC(S)$ (the curve complex of $S$).  Thus, for any fixed curve $\delta$ and $j$ sufficiently large,  Theorem~\ref{thm : bddgeod} implies that $d_{\gamma_j}(\gamma_0,\delta) \asya 0$. 
To see this, note that for $j$ sufficiently large the curve complex distance between $\gamma_j$ and every curve on a geodesic connecting $\gamma_0$ to $\delta$ is at least $3$ and hence $\gamma_j$ intersects all curves on the connecting geodesic. Thus Theorem~\ref{thm : bddgeod} implies a uniform bound on $d_{\gamma_j}(\gamma_0,\delta)$.
 Since each $\beta \in P_k$ is within distance $1$ of $\gamma_k$, similarly we have $d_{\beta}(\gamma_0,\delta) \asya 0$ for all $\beta \in P_k$, once $k$ is sufficiently large.

Suppose that $\{s_k\}_k$ satisfies (C1).  The filling set of curves $\gamma_{k-m+1},\ldots,\gamma_{k+m}$ have bounded length in $r(s_k)$.  So, for all $k$ sufficiently large and $\beta \in P_k$
\[ \tw_{\beta}(\delta,s_k) \asya \tw_{\beta}(\gamma_0,s_k)  \asya \diam_{\cC(\beta)}(\gamma_0 \cup \gamma_{k-m+1} \cup \cdots \cup \gamma_{k+m}).\]
If $\beta \in P_k^c$, then $\beta \not \in \{\gamma_j\}_j$, and so the term on the right is uniformly close to $0$ by Theorem~\ref{thm : subsurface coeff estimate}.  If $\beta = \gamma_j \neq \gamma_k$, then $j > k$, and the only curves in the set $\gamma_{k-m+1},\ldots,\gamma_{k+m}$ which actually intersect $\gamma_j$ nontrivially must have index {\em less} than $j$.  In this case, Theorem~\ref{thm : subsurface coeff estimate} implies that the term on the right is also uniformly close to $0$.
When $\beta = \gamma_k$, again appealing to Theorem~\ref{thm : subsurface coeff estimate}, the right-hand side is estimated (up to a bounded additive error) by
\[ d_{\gamma_k}(\gamma_0,\gamma_{k+m}) \asya e_k.\]
This proves the lemma when $\{s_k\}_k$ satisfies (C1).  The proof when $\{s_k\}_k$ satisfies (C2) is nearly identical since the curves 
$\gamma_{k-m+2},\ldots,\gamma_k,\gamma_{k+2},\ldots,\gamma_{k+m}$ have bounded length and fill $S \setminus \gamma_{k+1}$, so the only curve whose twisting we can no longer estimate is $\gamma_{k+1}$.  Since the conclusion of the lemma is silent regarding the twisting about this curve in case (C2), we are done.
\end{proof}

\begin{proof}[Proof of Theorem~\ref{thm : main reduction}]
In either case that $\{s_k\}_k$ satisfies (C1) or (C2), define $x_k = w_{s_k}(\gamma_k) + \tw_{\gamma_k}(\gamma_0,s_k) \ell_{\gamma_k}(s_k)$.  Observe that $\ell_{\gamma_k}(s_k) \asym 1$, so by (\ref{eqn : collar lemma}) $w_{s_k}(\gamma_k) \asya 1$ (these estimates depend on $\epsilon$ in case (C1), but not $k$).  Moreover, from Lemma~\ref{lem : sufficient twist estimates}, for any curve $\delta$ and $k$ sufficiently large we have
\[ \tw_{\gamma_k}(\gamma_0,s_k) \asya \tw_{\gamma_k}(\delta,s_k) \asya e_k \to \infty\]
as $k \to \infty$.
Consequently, we have $x_k\asym e_k$ and
\begin{equation} \label{eqn : gamma_k winner} \frac{\ell_{\delta}(s_k,\gamma_k)}{x_k\I(\delta,\gamma_k)}= \frac{w_{s_k}(\gamma_k) + \tw_{\gamma_k}(\delta,s_k) \ell_{\gamma_k}(s_k)}{w_{s_k}(\gamma_k) + \tw_{\gamma_k}(\gamma_0,s_k)\ell_{\gamma_k}(s_k)} \to 1
\end{equation}
as $k \to \infty$. Combining this with (\ref{eqn : intersection estimates end}) and using the setup of integers $A(0,k)$, for large $k$, we have
\begin{equation} \label{eqn : gamma_k contribution C1} \ell_{\delta}(s_k,\gamma_k) \asym x_k\I(\delta,\gamma_k) \asym e_k\I(\delta,\gamma_k) \asym A(0,k+m) \end{equation}

Now suppose that we are in case (C1) and $\beta_k \in P_k$, but $\beta_k \neq \gamma_k$.  As for $\gamma_k$ above, we have $\ell_{\beta_k}(s_k) \asym 1 \asya w_{s_k}(\beta_k)$ (with errors depending on $\epsilon$, but not $k$).  Combining this with (\ref{eqn : intersection estimates end}) and Lemma~\ref{lem : sufficient twist estimates}, we have
\[ \ell_{\delta}(s_k,\beta_k) \asym\I(\delta,\beta_k).\]
Therefore, by (\ref{eqn : gamma_k contribution C1}) and Lemma~\ref{lem : asymptotic intersection off gamma-k}, we have
\begin{equation} \label{eqn : beat by A(0,k+m)} \frac{\ell_{\delta}(s_k,\beta_k)}{x_ki(\delta,\gamma_k)} \asym \frac{i(\delta,\beta_k)}{A(0,k+m)} \to 0 \end{equation}
as $k \to \infty$.

Combining Theorem~\ref{thm : CRS} with Lemma~\ref{lem : asymptotic intersection off gamma-k}, (\ref{eqn : gamma_k winner}), and (\ref{eqn : beat by A(0,k+m)}), for any curve $\delta$ we have
\[ \lim_{k \to \infty} \tfrac{\ell_{\delta}(s_k)}{x_k \I(\delta,\gamma_k)}  = \lim_{k \to \infty} \tfrac{\ell_{\delta}(s_k,\gamma_k)}{x_k\I(\delta,\gamma_k)}
+ \tfrac{1}{x_k\I(\delta,\gamma_k)}\Big( \sum_{\substack{\beta_k \in P_k\;\text{and}\\ \beta_k \neq \gamma_k}} \!\!\! \ell_{\delta}(s_k,\beta_k) + O\Big(\sum_{\beta_k \in P_k}\I(\delta,\beta_k)\Big) \Big) = 1, \]
as required.

When $\{s_k\}_k$ satisfies (C2), $x_k$ is defined as above, and we define
\[ y_k = w_{s_k}(\gamma_{k+1}) + \tw_{\gamma_{k+1}}(\gamma_0,s_k)\ell_{\gamma_{k+1}}(s_k).\]  
According to Corollary~\ref{cor : lengths along r}, $\ell_{\gamma_{k+1}}(s_k) \to 0$ as $k \to \infty$, and so by (\ref{eqn : collar lemma}) we have
\[ w_{s_k}(\gamma_{k+1}) \to \infty \]
as $k \to \infty$.  Moreover, Lemma~\ref{lem : sufficient twist estimates} ensures that for any curve $\delta$ and $k$ sufficiently large, we have
\[ \tw_{\gamma_{k+1}}(\gamma_0,s_k) \asya \tw_{\gamma_{k+1}}(\delta,s_k).\]
Therefore, 
\[ \lim_{k \to \infty} \frac{\ell_{\delta}(s_k,\gamma_{k+1})}{y_k\I(\delta,\gamma_{k+1})} = 1,\]
and combining this with (\ref{eqn : gamma_k winner}) we have
\[ \lim_{k \to \infty} \frac{\ell_{\delta}(s_k,\gamma_k) + \ell_{\delta}(s_k,\gamma_{k+1})}{x_k\I(\delta,\gamma_k) + y_k\I(\delta,\gamma_{k+1})} = 1.\]
Because the estimate (\ref{eqn : beat by A(0,k+m)}) still holds for any curve $\beta_k\in P_k$ where $\beta_k \neq \gamma_k$ or $\gamma_{k+1}$, we can again apply Theorem~\ref{thm : CRS} and Lemma~\ref{lem : asymptotic intersection off gamma-k} to deduce that for any curve $\delta$ we have
\[ \lim_{k \to \infty} \frac{\ell_{\delta}(s_k)}{x_k\I(\delta,\gamma_k) + y_k\I(\delta,\gamma_{k+1})} = \lim_{k \to \infty}  \frac{\ell_{\delta}(s_k,\gamma_k) + \ell_{\delta}(s_k,\gamma_{k+1})}{x_k\I(\delta,\gamma_k) + y_k\I(\delta,\gamma_{k+1})} = 1,\]
completing the proof in case (C2), and hence in general.
\end{proof}

%%%%%%%%%%%%%%%%%%
%%%%%%%%%%%%%%%%%%%%
\section{The uniquely ergodic case}\label{sec : UE case}
%%%%%%%%%%%%%%%%%%%
%%%%%%%%%%%%%%%%%

Let $r:[0,\infty)\to\Teich(S)$ be a WP geodesic ray, and denote the ending lamination of $r$ by $\nu$; see $\S$\ref{subsec : endlam}. The following immediately implies Theorem~\ref{thm : UE main vague} from the introduction.
\begin{thm}  \label{thm : UE precise}
Suppose that $\nu$ is uniquely ergodic, then the limit set of $r$ in $\PML(S)$ (Thurston boundary) is the point $[\bar{\nu}]$. 
\end{thm}
\begin{proof}
The proof of the theorem closely follows Masur's  proof of the analogous fact for Teichm\"{u}ller geodesics \cite[Theorem 1]{2bdriesteich}.  Assuming $[\bar \xi]$ is any accumulation point of $r$, let $\{t_i\}_i$ be a sequence of times so that $r(t_i) \to [\bar \xi]$ as $i \to \infty$ in the Thurston compactification of Teichm\"{u}ller space $\Teich(S) \cup \PML(S)$; see $\S$\ref{subsec : Thcpct}.  According to the {\em Fundamental Lemma} of \cite[expos\'e 8]{FLP}, there exists a sequence $\{\bar \mu_i\}_i \subset \ML(S)$, such that
\[ \I(\bar \mu_i,\delta) \leq \ell_{\delta}(r(t_i)), \]
for all simple closed curves $\delta$, as well as a sequence of positive real numbers $\{b_i\}_i$, so that $b_i \bar \mu_i \to \bar \xi \in \ML(S)$ and $b_i \to 0$, as $i \to \infty$.

Let $\bar \nu$ be a transverse measure on $\nu$.  Since $\nu$ is uniquely ergodic, and hence minimal, there exists a sequence of Bers curves $\{\gamma_i\}_i$ and positive real numbers $\{c_i\}_i$, so that $c_i\gamma_i \to \bar{\nu} \in \mathcal{ML}(S)$ and $c_i \to 0$, as $i \to \infty$.
Since the $\gamma_i$ are Bers curves, by the inequality above there exists $C > 0$ so that
\[ \I(\bar \mu_i,\gamma_i) \leq \ell_{\gamma_i}(r(t_i)) \leq C.\]
Consequently, by continuity of the intersection form, we have
\[ \I(\bar \xi,\bar \nu) = \lim_{i \to \infty} \I(b_i\bar \mu_i,c_i\gamma_i) = \lim_{i \to \infty} b_ic_i \I(\bar \mu_i,\gamma_i) \leq C \lim_{i \to \infty} b_i c_i = 0,\]
and hence $\I(\bar \xi,\bar \nu) = 0$.  Because $\nu$ is uniquely ergodic, \cite[Lemma 2]{2bdriesteich} implies $\bar{\xi}$ is a multiple of $\bar{\nu}$, and therefore $[\bar{\xi}]=[\bar{\nu}]$. Since $[\bar{\xi}]$ was an arbitrary accumulation point of $r$ in $\PML(S)$, the proof is complete. 
\end{proof}

%%%%%%%%%%%%%
\section{Appendix}
%%%%%%%%%%%%

In this appendix we provide the proofs of the results of \S\ref{subsec : seq of curves} about sequences of curves.  As we mentioned there, many of the proofs closely follow the ones in \cite{nue2}, while others have been streamlined since the writing of that paper.  Here we mainly outline the proofs that are similar, incorporating the required changes, and otherwise provide the streamlined proofs.

\subsection{Subsurface coefficient estimates}
In the next Lemma, $B_0$ is the constant from Theorem~\ref{thm : beh ineq}.
\begin{lem} \label{lem : local to global} \textnormal{(Local to Global)}
Fix any $B  \geq B_0+1$, and let $\{\delta_k\}_{k=0}^\omega$ ($\omega \in\mathbb{Z}^{\geq 0}\cup\{\infty\}$) be a (finite or infinite) sequence of curves in $\cC(S)$, with the property that $\delta_{k-1}\pitchfork \delta_k,\delta_{k+1}\pitchfork \delta_k$ and that $d_{\delta_k}(\delta_{k-1},\delta_{k+1}) \geq 3B$ for all $k \geq 1$.  Then for all $0 \leq i < k <j$, we have that $\delta_i\pitchfork \delta_k,\delta_j\pitchfork \delta_k$ and that
\begin{equation} \label{eq : LtoG}
|d_{\delta_k}(\delta_i,\delta_j) - d_{\delta_k}(\delta_{k-1,},\delta_{k+1})| \leq 2B_0.
\end{equation}
\end{lem}
\begin{proof} To simplify the notation, write $d_k(i,j)= d_{\delta_k}(\delta_i,\delta_j)$.
The proof is by induction on $n = j-i$.  The base case is $n = 2$, in which case $i = k-1$, $j= k+1$, and the conclusions of the lemma hold trivially.

We suppose that $\gamma_i\pitchfork \gamma_k$ and that (\ref{eq : LtoG}) holds for all $i,j$ with $i < k < j$ and $j-i \leq n$, and prove them for $n +1$.  To that end, suppose that $0 \leq i < k < j$ are such that $j-i = n+1$.  We claim  that $d_k(i,k-1) \leq B_0$.  To see this, note that if $i = k-1$, then the claim holds obviously.  Otherwise,  $i < k-1 < k$ and $k-i \leq n$, so by hypothesis of the induction $\gamma_i\pitchfork \gamma_k$ and
\[ d_{k-1}(i,k) \geq e_k \geq 3B-2B_0 > B_0,\]
note that $\delta_{k-1}\pitchfork \delta_k$ by assumptions of the lemma. Then Theorem~\ref{thm : beh ineq} implies that $d_k(i,k-1) \leq B_0$.  By a similar reasoning, we have that $d_k(k+1,j) \leq B_0$, and so by the triangle inequality
\[ |d_k(i,j) - d_k(k-1,k+1)| \leq d_k(i,k-1) + d_k(j,k+1) \leq 2B_0,\]
which is (\ref{eq : LtoG}). Moreover, since $d_k(k,k+1)\geq 3B_0$ for the above inequality we have that $d_k(i,j)\geq B_0+1> 1$ which implies that $\gamma_i\pitchfork\gamma_j$, finishing the proof of lemma by induction.
\end{proof}

Set $B = \max\{3,B_0 + 1,G_0\}$, where $G_0$ is the constant from Theorem~\ref{thm : bddgeod} for a geodesic in $\mathcal C(S)$.  Set $E_0 = 3B + 4$, and for the remainder of this subsection assume that the sequence $\Gamma(\mathcal E) = \{\gamma_k\}_{k=0}^\infty$ satisfies $\mathcal P(\mathcal E)$ from Definition~\ref{def : sequence of curves}, where $\mathcal E = \{e_k\}_{k=0}^\infty$, $e_k \geq a e_{k-1}$ for some $a \geq 1$ and all $k$, and $e_0 \geq E_0$ (and hence $e_k \geq E_0$ for all $k$).
Also throughout this subsection let $\mathbb{M}$ be the monoid generated by $\{m,m+1\}$. A simple arithmetic computation shows that any integer which is greater than or equal to $ m^2-1$ is in $\mathbb{M}$.

\begin{lem} \label{lem : nearly partial order} For all $i < k < j$ such that $k-i,j-k \in \mathbb M$, e.g.~if $k-i,j-k \geq m^2-1$, we have that $\gamma_i\pitchfork \gamma_k,\gamma_j\pitchfork \gamma_k$ and that
\[|d_{\gamma_k}(\gamma_i,\gamma_j)  - e_k| \leq 2B_0 + 4.\]
\end{lem}
\begin{proof}  As in the previous proof, we write $d_k(i,j) = d_{\gamma_k}(\gamma_i,\gamma_j)$, and also write $\I(i,j) = \I(\gamma_i,\gamma_j)$ and $\pi_i(j) = \pi_{\gamma_i}(\gamma_j)$.  

We make a few observations from Definition~\ref{def : sequence of curves}.  First, $\I(j,j+1) = 0$ for all $j$, and hence if $\pi_k(j),\pi_k(j+1) \neq \emptyset$ for some $k$, then $d_k(j,j+1) = 1$.  Second, for all $k$, $\I(k,k+m),\I(k,k+m+1) \neq 0$.  Consequently, $\pi_k(j) \neq \emptyset$ for all $j,k$ with $|j-k| \in \{m,m+1\}$.  Finally, observe that $|d_k(k-m,k+m) - e_k| \leq 2$.  Thus, if $i < k < j$ and $k-i,j-k \in \{m,m+1\}$,  by the triangle inequality
\begin{equation} \label{eqn : big enough twist} |d_k(i,j) - e_k | \leq 4.
\end{equation}

Now, for any sequence of integers $\{k_j\}_j$ such that $k_{j+1}-k_j \in \{m,m+1\}$ for all $j$, the sequence $\{\gamma_{k_j}\}_j$ has the property that $\gamma_{k_{j-1}},\gamma_{k_{j+1}}\pitchfork \gamma_{k_j}$ and that $d_{k_j}(k_{j-1},k_{j+1}) \geq 3B$ by (\ref{eqn : big enough twist}) and since $e_k\geq E_0$. Hence the sequence $\{\gamma_{k_j}\}_j$ satisfies the assumptions of Lemma~\ref{lem : local to global}. Then by the lemma, for any $\ell$ with $i < \ell < j$, $\gamma_\ell\pitchfork\gamma_i, \gamma_\ell\pitchfork\gamma_j$ and 
\[ |d_{k_\ell}(k_i,k_j) - e_{k_\ell}| \leq 2B_0 + 4 .\]

Therefore, if $i < k <j$ and $k-i,j-k \in \mathbb M$, then
\[ |d_k(i,j) - e_k| \leq 2B_0 + 4.\]
 This completes the proof of the inequality in the statement of the lemma. 
\end{proof}
\begin{lem} \label{lem : good quasi-geodesic} The map $k \mapsto \gamma_k$ is a $1$--Lipschitz, $(K,C)$--quasi-geodesic, where $K = C = 2m^2 + 2m -1$.
\end{lem}
\begin{proof}  First, suppose that $i < j$ with $j-i \geq 2m^2 + 2m - 1$.  Then by Lemma~\ref{lem : nearly partial order} for each $k \in\{ i+m^2,i+m^2+1,\ldots,j-m^2\}$, we have that $\gamma_k\pitchfork \gamma_i, \gamma_k\pitchfork \gamma_j$ and that
\begin{equation} \label{eqn : twisting to fill} d_{\gamma_k}(\gamma_i,\gamma_j) \geq e_k - 2B_0 - 4 \geq B \geq 3. \end{equation}
Thus the curves $\gamma_i,\gamma_j$ fill the annulus with core curve $\gamma_k$. This implies that any curve that intersects $\gamma_k$ must intersect one of $\gamma_i$ or $\gamma_j$. Moreover, the $2m$ curves $\gamma_k$ for $k=i+m^2,\ldots,i+m^2+2m-1$ fill $S$, $\gamma_i,\gamma_j$ also fill $S$.

Next, suppose that $j>i +2m^2+2m-1$ and write $j = i + q(2m^2+2m-1) + r$, where $q,r$ are nonnegative integers with $0 \leq r < 2m^2+2m-1$.  Set the curve $\delta_k = \gamma_{i+k(2m^2 + 2m-1)}$, for $k = 1,\ldots, q-1$.  Then the curves
\[ \gamma_i, \delta_1, \ldots, \delta_{q-1}, \gamma_j \]
form a sequence in $\mathcal C(S)$. As we saw above any two distinct curves in the sequence fill $S$, and by (\ref{eqn : twisting to fill}) for all $0 < k < q$ we have
\[ d_{\delta_k}(\gamma_i,\gamma_j) \geq B > G_0.\]
So, by Theorem~\ref{thm : bddgeod}, a geodesic from $\gamma_i$ to $\gamma_j$ must have a vertex disjoint from $\delta_k$, for all $k = 0,\ldots,q$.  Since any two curves $\delta_k,\delta_{k'}$ in the sequence fill $S$, no curve can be disjoint from more than one of them, and hence the geodesic must contain at least $q+1$ vertices, so
\[ d(\gamma_i,\gamma_j) \geq q = \frac{j-i - r}{2m^2 + 2m-1} \geq \frac{1}{K}(j-i - C) \]
where $K = 2m^2 + 2m -1$ and $C = 2m^2 +2m-1$.  Since this inequality trivially holds if $j - i < C$ and $i < j$, the required lower bound follows.  Moreover, since $\gamma_k,\gamma_{k+1}$ are disjoint, the map is $1$--Lipschitz. Finally the upper bound for the $(K,C)$-quasi-geodesic is immediate.
\end{proof}

For each $k \geq 0$, let $\mu_k := \{\gamma_k,\ldots,\gamma_{k+2m-1} \}$. 

\begin{lem} \label{lem : soft projection bound} There exists $M > 0$ such that for any subsurface $W \subsetneq S$ which is neither $S$ nor an annulus with core curve some $\gamma_k$, we have
\[ d_W(\mu_i,\mu_j) \leq M \]
for all $i,j$.
\end{lem}
\begin{proof}  First, let $\mu_k' = \{\gamma_k,\ldots,\gamma_{k+2m-2},\gamma_{k+2m-1}'\}$ where $\gamma_{k+2m-1}'$ is as in Definition~\ref{def : sequence of curves}.  From the definition, any curve in $\mu_k$ and curve in $\mu_{k+1}'$ have uniformly bounded intersection number (bounded by $m^2b_2$).  Consequently, there exists $M_0 > 0$ such that for any subsurface $W$ and any $k$, we have
\[ d_W(\mu_k,\mu_{k+1}') \leq M_0.\]

Next, observe that $\mu_{k+1}'$ and $\mu_{k+1}$ differ by Dehn twisting $\gamma_{k+2m}'$ about $\gamma_{k+m}$ (which has zero intersection number with all curves in $\mu_{k+1}'$ except $\gamma_{k+2m}'$).  Therefore, there exists another constant $M_1 >0$ so that as long as $W$ is not the annulus with core $\gamma_{k+m}$, we have
\[ d_W(\mu_{k+1}',\mu_{k+1}) \leq M_1. \]
Indeed, for any such $W$, $\pi_W(\mu_{k+1}') \cap \pi_W(\mu_{k+1}) \neq \emptyset$, and so we may take $M_1$ to be at most the sum of the diameters of $\pi_W(\mu_{k+1}')$ and $\pi_W(\mu_{k+1}')$ which is at most $4$.

From these two inequalities, we see that for any subsurface $W$ which is not the annulus with core $\gamma_{k+m}$, the triangle inequality implies
\[ d_W(\mu_k,\mu_{k+1}) \leq M_0 + M_1.\]
From this it follows that for any $D > 0$ and $|j-i| \leq D$, 
\begin{equation} \label{eqn : lipschitz bound}  d_W(\gamma_i,\gamma_j) \leq D(M_0 + M_1) \end{equation}
whenever $W$ is not an annulus with core curve $\gamma_k$, for some $k$.

Finally, suppose that $W$ is any subsurface which is not $S$ and not an annulus with core curve $\gamma_k$, for some $k$.  By Lemma~\ref{lem : good quasi-geodesic}, $k \mapsto \gamma_k$ is a quasi-geodesic in $\cC(S)$. So there is a uniform bound for the number its vertices that are within distance $1$ of $\partial W$.  Consequently, there exists $D_0 > 0$ (independent of $W$) and $i_0$ so that if $k \not \in [i_0,i_0+D_0]$, then $\pi_W(\gamma_k) \neq \emptyset$.  By Theorem~\ref{thm : bddgeod}, there exists $G = G(K,C)$ so that the projections of both sequences $\{\gamma_k\}_{k=0}^{i_0}$ and $\{\gamma_k\}_{k=i_0+D_0}^\infty$ to $W$ have diameter at most $G$.  Combining this with (\ref{eqn : lipschitz bound}) and setting $M = 2G + D_0(M_0 + M_1)$, we have
\[ \diam_{\cC(W)}(\{\gamma_k\}_{k=0}^\infty) \leq 2G + D_0(M_0 + M_1) = M.\]
Since $M$ is independent of the subsurface $W$, this completes the proof.
\end{proof}

\begin{proof}[Proof of Theorem \ref{thm : subsurface coeff estimate}]
The fact that $\{\gamma_k\}_k$ is a $1$--Lipschitz $(K,C)$--quasi-geodesic is Lemma~\ref{lem : good quasi-geodesic}. Klarreich's work \cite[Theorem 4.1]{bdrycc}) describing the Gromov boundary of the curve complex then implies that there exists a $\nu\in\EL(S)$, so that any accumulation point of $\{\gamma_k\}_k$ in $\PML(S)$ is supported on $\nu$.

Any accumulation point of $\{\gamma_k\}_k$ in the Hausdorff topology of closed subset of $S$ contains $\nu$, and hence for a subsurface $W\subseteq S$, $\pi_W(\nu) \subseteq \pi_W(\gamma_k)$ for all $k$ sufficiently large.  Consequently the equations on the left of (\ref{eqn : gamma_i, big projection}) and (\ref{eqn : not gamma_i, small projection}) follow from Lemmas~\ref{lem : nearly partial order} and~\ref{lem : soft projection bound}, respectively, setting $R = \max\{M,2B_0 + 4\}$. For any marking $\mu$, the pairwise intersection between curves in $\mu$ and in $\mu_0$ are bounded by some finite number, and hence $d_W(\mu,\mu_0)$ is uniformly bounded by some constant $D > 0$, independent of $W$.  Setting $R(\mu) = R + D$, the equations on the right-hand side of (\ref{eqn : gamma_i, big projection}) and (\ref{eqn : not gamma_i, small projection}) then follow from those on the left-hand side, together with the triangle inequality.
\end{proof}

\subsection{Intersection number estimates}

We now assume that $\mathcal E = \{e_k\}_k$ grows exponentially, with $e_k \geq a e_{k-1}$ for some $a > 1$ and all $k\geq 1$.  The aim is to estimate intersection numbers in terms of the numbers $A(i,k)$ defined in (\ref{eq : A(i,k)}).  We begin with the upper bound.

\begin{lem} \label{lem : upper bound on intersection numbers}  If $\Gamma(\mathcal E) = \{\gamma_k\}_k$ satisfies $\mathcal P(\mathcal E)$ with $e_k \geq a e_{k-1}$ for some $a > 1$, then there exists $\kappa > 0$ such that
\[ \I(\gamma_i,\gamma_k) \leq \kappa A(i,k).\]
Moreover, we may take $\kappa$ to be decreasing as a function of $a$.
\end{lem}
\begin{proof}[Sketch of proof]
The proof is a rather complicated induction, but is essentially identical to the proof of Proposition~5.5 from \cite{nue2}.  We sketch the proof for completeness.

We first recall from \cite[expos\'{e} 4]{FLP}, that for any simple closed curves $\beta,\delta,\delta'$ and integer $e$, we have
\begin{equation} \label{eqn : FLP intersection number} \Big| \I(D_\beta^e(\delta'),\delta) - |e|\I(\delta',\beta)\I(\delta,\beta)\Big| \leq \I(\delta,\delta').
\end{equation}
Since $\gamma_{k+m} = D_{\gamma_k}^{e_k}(\gamma_{k+m}')$ and $\I(\gamma_{k+m}',\gamma_k) = b$, we can apply this to estimate $\I(\gamma_i, \gamma_{k+m})$ to obtain
\[ \Big| \I(\gamma_i,\gamma_{k+m}) - be_k \I(\gamma_i,\gamma_k) \Big| \leq \I(\gamma_i,\gamma_{k+m}').\]
The right-hand side can be bounded as follows.  Since the curves $\gamma_{k-m},\ldots,\gamma_{k+m-1}$ fill $S$, these cut any simple closed curve $\delta$ into $N$ arcs, where
\[ N = \sum_{j=k-m}^{k+m-1} \I(\delta,\gamma_j).\]
Now apply this cutting procedure to both $\gamma_{k+m}'$ and $\gamma_i$.  Any pair of resulting arcs (one from $\gamma_{k+m}'$ and one from $\gamma_j$) are either disjoint, intersect at most once if they lie in a complementary disk, or intersect at most twice if they lie in a once-punctured complementary disk.  Thus
\[\I(\gamma_i,\gamma'_{k+m})\leq 2\sum_{j=k-m}^{k+m-1} \I(\gamma_i,\gamma_j)\sum_{j=k-m}^{k+m-1}\I(\gamma'_{k+m},\gamma_j).\]
Moreover, the assumption on the values of intersection numbers of $\gamma_{k+m}'$ and the curves $\gamma_j, \; j=k-m,\ldots,k+m-1$ from Definition~\ref{def : sequence of curves} implies that 
\[ \sum_{j=k-m}^{k+m-1} \I(\gamma_j,\gamma_{k+m}') \leq (m+1)b'.\]
Consequently, setting $B = 2(m+1)b'$, we have
\begin{equation} \label{eqn : key intersection estimate from FLP} \Big| \I(\gamma_i,\gamma_{k+m}) - be_k \I(\gamma_i,\gamma_k) \Big| \leq B \sum_{j = k-m}^{k+m-1} \I(\gamma_i,\gamma_j),
\end{equation}
and hence
\[ \I(\gamma_i,\gamma_{k+m}) \leq be_k \I(\gamma_i,\gamma_k) + B \sum_{j=k-m}^{k+m-1} \I(\gamma_i,\gamma_j).\]

The goal is to show that $\I(\gamma_i,\gamma_{k+m}) \leq \kappa A(i,k+m)$ for some $\kappa > 0$.  The proof is by induction on $(k+m)-i$, and the constant $\kappa$ is actually a limit of an increasing sequence of constants $K(1) <K(2) < K(3) < \ldots$.  To see what these constants should be, we assume by induction that $\I(\gamma_i,\gamma_j) \leq K(j-i) A(i,j)$ for all $i < j$ with $j-i < k+m$, then dividing both sides of the inequality above by $A(i,k+m)$, we have
\begin{eqnarray*} \frac{\I(\gamma_i,\gamma_{k+m})}{A(i,k+m)} & \leq & \frac{be_k \I(\gamma_i,\gamma_k)}{A(i,k+m)} + B\sum_{j=k-m}^{k+m-1} \frac{\I(\gamma_i,\gamma_j)}{A(i,k+m)} \\
& \leq & K(k-i) \frac{be_kA(i,k)}{A(i,k+m)} + B \sum_{j= k-m}^{k+m-1} K(j-i) \frac{A(i,j)}{A(i,k+m)} \\
& \leq & K(k+m-1-i) \Big( 1 + B \sum_{j=k-m}^{k+m-1} a^{-\lfloor \frac{k - i}{m} \rfloor} \Big)\\
& = & K((k+m-i)-1)(1 + 2mB a^{-\lfloor \frac{k-i}{m} \rfloor} ) \\
\end{eqnarray*}
The right-hand side above suggests the recursive/inductive definition
\[ K(k+m-i) := K((k+m-i)-1)(1 + 2mB a^{-\lfloor \frac{k-i}{m} \rfloor}).\] 
(Note that the right-hand side depends only on the difference $k-i$ and not on $i$ and $k$ independently).
For $a > 1$, one shows (using a comparison with a geometric series, after taking logarithms) that $K(j)$ so defined is bounded, and since it is increasing, it converges to a constant $\kappa > 0$.  Since $K(j) < \kappa$ for all $j$, the inequality above proves the lemma.  See \cite[Proposition~5.5]{nue2} for the details.
\end{proof}

Since we are assuming fewer nonzero intersection numbers in the current work than in \cite{nue2}, the lower bounds we obtain are weaker, and the proof is slightly more complicated than the one in $\S 5$ of \cite{nue2}.  Fortunately, the weaker estimates suffice for our purposes.

\begin{lem} \label{lem : lower bound on intersection numbers}  If $e_k \geq a e_{k-1}$ for $a > 1$ sufficiently large and all $k\geq 0$, then there exists $\kappa' > 0$ such that
\[ \I(\gamma_i,\gamma_k) \geq \kappa' A(i,k) \]
whenever
\begin{enumerate}
\item $0 \leq i \leq 2m-1$ and $k \geq i + m^2 + m -1$, or
\item $k-i \geq 2m$ and $i \equiv k$ mod $m$.
\end{enumerate}
\end{lem}
\begin{proof}  Fix some constant $a_0 > 1$ and assume to begin that $e_k \geq a_0 e_{k-1}$, and let $\kappa > 0$ be the constant from Lemma~\ref{lem : upper bound on intersection numbers}.  We will eventually take a larger $a > a_0$, but we observe that upper bound on intersection numbers from Lemma~\ref{lem : upper bound on intersection numbers} remains valid with this choice of $\kappa$.

Now, recalling that $A(i,k+m) = be_kA(i,k)$, dividing (\ref{eqn : key intersection estimate from FLP}) by $A(i,k+m)$, we have
\[ \frac{\I(\gamma_i,\gamma_{k+m})}{A(i,k+m)} \geq \frac{\I(\gamma_i,\gamma_k)}{A(i,k)} - B \sum_{j = k-m}^{k+m-1} \frac{\I(\gamma_i,\gamma_j)}{A(i,k+m)}. \]
Combining this with Lemmas~\ref{lem : intersection estimate ratios} and \ref{lem : upper bound on intersection numbers} we have
\begin{eqnarray*}
\frac{\I(\gamma_i,\gamma_{k+m})}{A(i,k+m)} & \geq & \frac{\I(\gamma_i,\gamma_k)}{A(i,k)} - \kappa B \sum_{j=k-m}^{k+m-1} \frac{A(i,j)}{A(i,k+m)}\\
& \geq & \frac{\I(\gamma_i,\gamma_k)}{A(i,k)} - 2m\kappa B a^{-\lfloor \frac{k-i}m \rfloor}
\end{eqnarray*}
Recursively substituting, we see that for any $n \geq 0$ such that $i < k-nm$, 
\begin{equation}\label{eq : I/A}
\frac{\I(\gamma_i,\gamma_{k+m})}{A(i,k+m)} \geq \frac{\I(\gamma_i,\gamma_{k-nm})}{A(i,k-nm)} - 2m\kappa B \sum_{s=0}^n a^{-\lfloor \frac{k-i}m \rfloor + s}
\end{equation}
In this situation, taking $a>a_0 > 1$ sufficiently large, we can make the term
\[ 2m\kappa B \sum_{s=0}^n a^{-\lfloor \frac{k-i}m \rfloor + s} \]
as small as we like, independent of $n$ (since this is a partial sum of a geometric series with common ratio $a$).  Specifically, we choose $a > a_0 > 1$ large enough so that for all $k-i \geq m$ and $n \geq 0$, the sum is bounded above by $\kappa'$, where
\[ \kappa' = \frac12 \min\Big\{\frac{1}{A(i,j)} \mid 0 \leq i \leq 2m-1 \mbox{ and } i < j \leq m^2 + 3m -2 \Big\}. \]

Now, suppose that $0 \leq i \leq 2m-1$ and that $k \geq i + m^2 + m -1$, and write $k = j + m$ (so $j \geq i + m^2 -1$).  Let $n \geq 0$ be such that
\[ i + m^2 -1 \leq j - nm \leq m^2 + 3m -2,\]
which is possible since $i + m^2 -1 \leq m^2 + 2m -2$.  By Lemma~\ref{lem : nearly partial order}, $\gamma_i\pitchfork\gamma_{j-nm}$, so $\I(\gamma_i,\gamma_{j-nm}) \geq 1$, and therefore
\[
\frac{\I(\gamma_i,\gamma_{j+m})}{A(i,j+m)} \geq \frac{\I(\gamma_i,\gamma_{j-nm})}{A(i,j-nm)} - 2m\kappa B \sum_{s=0}^n a^{-\lfloor \frac{j-i}m \rfloor + s} \geq 2 \kappa' - \kappa' = \kappa'.
\]
That is, $\I(\gamma_i,\gamma_k) \geq \kappa' A(i,k)$, proving part (i).

Next, let $k,i$ be any two positive integers such that $k - i \geq 2m$ and $i \equiv k$ mod $m$.  Write $k = j+m$ and let $n \geq 0$ be such that $j-nm = i+m$.  Then by (\ref{eq : I/A}) we have
\begin{eqnarray*}
\frac{\I(\gamma_i,\gamma_{j+m})}{A(i,j+m)} & \geq & \frac{\I(\gamma_i,\gamma_{j-nm})}{A(i,j-nm)} - 2m\kappa B \sum_{s=0}^n a^{-\lfloor \frac{j-i}m \rfloor + s} \\
& = & \frac{\I(\gamma_i,\gamma_{i+m})}{1} - \kappa' \geq 2 \kappa'  - \kappa' = \kappa'.
\end{eqnarray*}
The last inequality follows from the fact that $\kappa' \leq \frac12$ and $\I(\gamma_i,\gamma_{i+m}) = b \geq 1$.  This proves (ii), and completes the proof of the lemma.
\end{proof}

\begin{proof}[Proof of Theorem~\ref{thm : intersection}.] The first paragraph of the theorem follows immediately from Lemmas~\ref{lem : upper bound on intersection numbers} and \ref{lem : lower bound on intersection numbers}, setting $\kappa_0 = \max\{\kappa,\frac1{\kappa'}\}$.  The second paragraph is the next lemma.
\end{proof}

\begin{lem} \label{lem : intersection with delta estimate} For any curve $\delta$, there exists $\kappa(\delta) > 0$ and $N(\delta) \in \mathbb Z$ so that for all $k \geq N(\delta)$, we have
\[ \I(\delta,\gamma_k) \asym_{\kappa(\delta)} A(0,k).\]
\end{lem}
We only sketch the proof since this is exactly the same statement and proof as in Lemma~5.11 from \cite{nue2}.  
\begin{proof}[Sketch of proof]
The idea is that from the estimates in Lemma~\ref{lem : upper bound on intersection numbers} and \ref{lem : lower bound on intersection numbers}, we have some $\kappa_0,n_0 > 0$ so that for all $0 \leq i \leq 2m-1$ and $k \geq n_0$,
\[ \I(\gamma_i,\gamma_k) \asym_{\kappa_0} A(0,k). \]
Since $\mu_0 = \gamma_0 \cup \ldots \cup \gamma_{2m-1}$ fills $S$, this means that the set of laminations $\{\frac{\gamma_k}{A(0,k)} \}_{k \geq n_0} \subset \ML(S)$ form a compact subset of $\ML(S)$.  By Theorem~\ref{thm : intersection}, this sequence can only accumulate on points of $\ML(S)$ supported on $\nu$.  So, for any curve $\delta$, there is a compact neighborhood of the accumulation points and a number $\kappa(\delta) > 0$ on which intersection number with $\delta$ lies in the interval $\left[ \tfrac{1}{\kappa(\delta)},\kappa(\delta) \right]$.  But then for $k$ sufficiently large, $\tfrac{\gamma_k}{A(0,k)}$ is in this neighborhood and hence
\[ \frac{\I(\delta,\gamma_k)}{A(0,k)} = \Big(\delta,\frac{\gamma_k}{A(0,k)}\Big) \asym_{\kappa(\delta)} 1 \]
as required.
\end{proof}

\subsection{Convergence in $\ML(S)$}
Lemma~\ref{lem : intersection with delta estimate}, together with (\ref{eqn : FLP intersection number}), are the key ingredients in the proof of the following, which is identical to Lemma~5.13 from \cite{nue2}.
\begin{prop} \label{prop : convergence in ML} For each $h = 0,\ldots,m-1$,
\[ \lim_{i \to \infty} \frac{\gamma_{h+im}}{A(0,h+im)} = \bar \nu^h \]
in $\ML(S)$, where $\bar \nu^h$ is a measure supported on $\nu$.
\end{prop}
\begin{proof}[Sketch of proof.] Applying (\ref{eqn : FLP intersection number}) to estimate $\I(\gamma_{k+m},\delta)$ using the fact that $\gamma_{k+m} = D_{\gamma_k}^{e_k}(\gamma_{k+m}')$, and applying Lemma~\ref{lem : intersection with delta estimate} we may argue as in the proof of Lemma~\ref{lem : upper bound on intersection numbers} we have
\begin{eqnarray*} |\I(\gamma_{k+m},\delta) - be_k\I(\gamma_k,\delta) | & \leq & \I(\gamma_{k+m}',\delta) \leq B \sum_{l = k-m}^{k+m-1} \I(\gamma_l,\delta)\\
& \leq & B \kappa(\delta) \sum_{l = k-m}^{k+m-1} A(0,l). \end{eqnarray*}
Dividing both side by $A(0,k+m)$ and applying Lemma~\ref{lem : intersection estimate ratios} this implies
\[ \left| \I \Big(\frac{\gamma_{k+m}}{A(0,k+m)},\delta \Big) - \I \Big(\frac{\gamma_k}{A(0,k)},\delta \Big) \right| \leq 2mB \kappa(\delta) a^{- \lfloor \frac{k}{m} \rfloor}.\]
From this and a geometric series argument, we deduce that for all $h = 0,\ldots,m-1$, the sequence $\Big\{ \I\Big( \frac{\gamma_{h+im}}{A(0,h+im)},\delta \Big) \Big\}_{i=0}^\infty$ is a Cauchy sequence of real numbers, and hence converges.  By Lemma~\ref{lem : intersection with delta estimate}, the limit is nonzero, and since this is true for every simple closed curve $\delta$, the sequence $\Big\{\frac{\gamma_{h+im}}{A(0,h+im)} \Big\}_{i=0}^\infty$ converges to some $\bar \nu^h \in \ML(S)$, supported on $\nu$ by Theorem~\ref{thm : subsurface coeff estimate}. 
\end{proof}

The next lemma is the analog of Theorem~6.1 from \cite{nue2}.  The proof is essentially the same, but since the required intersection number estimates are weaker here, we sketch the proof nonetheless.
\begin{lem} For each $h,h' \in \{0,\ldots,m-1\}$ with $h \neq h'$, we have
\[ \lim_{i \to \infty} \frac{\I(\gamma_{h+im},\bar \nu^h)}{\I(\gamma_{h+im},\bar \nu^{h'})} = \infty.\]
Consequently, $\bar \nu^h$ is not absolutely continuous with respect to $\bar \nu^{h'}$.
\end{lem}
\begin{proof}[Sketch of Proof.] As in the proof of \cite[Theorem~6.1]{nue2}, it clearly suffices to prove that for $i$ sufficiently large, 
\[ \I(\gamma_h,\gamma_{h+(i+1)m})\I(\gamma_{h+im},\bar \nu^h) \asym 1 \]
and
\[  \lim_{i \to \infty} \I(\gamma_h,\gamma_{h+(i+1)m})\I(\gamma_{h+im},\bar \nu^{h'}) = 0. \]
Appealing to Proposition~\ref{prop : convergence in ML} to estimate $\bar \nu^h$ by $\tfrac{\gamma_{h+km}}{A(0,h+km)}$ for large $k \gg i$, and Theorem~\ref{thm : intersection}, we have
\[ \I(\gamma_0,\gamma_{h+(i+1)m}) \I(\gamma_{h+im},\bar \nu^h)  \asym  \tfrac{A(0,h+(i+1)m)A(h+im,h+km)}{A(0,h+km)} = 1 \]
The last equality here follows from a simple calculation using the formula (\ref{eq : A(i,k)}) for $A(i,j)$ (see the proof of \cite[Theorem~6.1]{nue2} for details).  The multiplicative error here can be made arbitrarily close to $\kappa_0^2$ (taking $k$ sufficiently large).

Similarly, we estimate $\bar \nu^{h'} = \tfrac{\gamma_{h'+km}}{A(0,h'+km)}$ and apply Theorem~\ref{thm : intersection} to obtain
\[ \I(\gamma_0,\gamma_{h+(i+1)m})\I(\gamma_{h+im},\bar \nu^{h'}) \lm \tfrac{A(0,h+(i+1)m)A(h+im,h'+km)}{A(0,h'+km)} \lm a^{-i} \]
The first multiplicative error can be made arbitrarily close to $\kappa_0^2$ (again by taking $k$ sufficiently large).  The second bound follows from a calculation and Lemma~\ref{lem : intersection estimate ratios}, with multiplicative error depending only on whether $h >h'$ or $h' > h$  (see \cite{nue2} for details).
\end{proof}
\begin{proof}[Proof of Theorem~\ref{thm : nue}.]  All that remains is to prove that $\bar \nu^0,\ldots,\bar \nu^{m-1}$ are ergodic measures.  At this point, the proof is identical to the proof of the analogous statement Theorem~6.7 from \cite{nue2}, appealing to the facts proved so far.  This proof involves a detailed analysis of Teichm\"uller geodesics, drawing specifically on results of Lenzhen-Masur \cite{lenmasdivergence} and the fourth author \cite{rshteich}.  As this would take us too far afield of the current discussion, we refer the reader to that paper for the details.
\end{proof}

\bibliographystyle{amsalpha}
\bibliography{reference}
\end{document}